\documentclass[aos,preprint]{imsart} 

\RequirePackage[OT1]{fontenc}
\RequirePackage{amsthm,amssymb,amsfonts,amsmath}
\RequirePackage[numbers]{natbib}
\RequirePackage[colorlinks,citecolor=blue,urlcolor=blue, pdffitwindow=false, pdfstartview={FitH}]{hyperref}
\usepackage{color}
\usepackage{amsmath}
\usepackage{amssymb}
\usepackage{amsfonts}
\usepackage{amsthm}
\usepackage{rotating}
\usepackage{dsfont}
\usepackage{array}
\usepackage{multirow}
\usepackage{multicol}

\usepackage{natbib}

\usepackage{graphicx}

\usepackage{tabularx}
\usepackage{longtable}
\usepackage{lscape}
\usepackage{float}

 \usepackage{hyperref}
 \usepackage{tikz}
\usetikzlibrary{arrows}
\newtheorem{theo}{Theorem}[section]
\newtheorem{cor}{Corollary}[section]
\newtheorem{lemma}{Lemma}[section]
\newtheorem{rem}{Remark}[section]
\newtheorem{ex}{Example}[section]
\newtheorem{prop}{Proposition}[section]
\newtheorem{defi}{Definition}[section]

\newcommand{\PT}{\ensuremath{\mathcal{P}_{\mathcal{T},Id}}}
\newcommand{\PTf}{\ensuremath{\mathcal{P}_{\mathcal{T},\psi}}}
\newcommand{\blabla}{\ensuremath{\mathfrak{M}_{mma}(\Omega,\kappa)}}
\newcommand{\bla}{{\widetilde{\mathfrak{M}}_{mma}(\Omega,\kappa)}}

\newcommand{\PP}{\ensuremath{\mathbb{P}} }
\newcommand{\RR}{\ensuremath{\mathbb{R}} }

\newcommand{\NN}{\ensuremath{\mathbb{N}} }

\newcommand{\EE}{\ensuremath{{\mathbb E}}}

\newcommand{\R}{{\cal R}}

\newcommand{\Be}{\mathcal{B}_{\epsilon}}

\newcommand{\Bej}{\mathcal{B}_{\epsilon,j}}

\arxiv{arXiv:submit/1106085}

\startlocaldefs
\numberwithin{equation}{section}
\theoremstyle{plain}

\endlocaldefs

\begin{document}

\begin{frontmatter}
\title{Classification with the nearest neighbor rule in general finite dimensional spaces}
\runtitle{Nearest neighbor rule on general finite dimension spaces}

\begin{aug}
\author{\fnms{S\'ebastien} \snm{Gadat}\ead[label=e1]{sebastien.gadat@math.univ-toulouse.fr}} \and
\author{\fnms{Thierry} \snm{Klein}\ead[label=e2]{thierry.klein@math.univ-toulouse.fr}} \and 
\author{\fnms{Cl\'ement} \snm{Marteau}\ead[label=e3]{clement.marteau@math.univ-toulouse.fr}}

\runauthor{S. Gadat, T. Klein, C. Marteau}

\affiliation{Toulouse School of Economics, Universit\'e Toulouse I Capitole}
\address{Toulouse School of Economics\\ 
 Universit\'e Toulouse 1 - Capitole.\\
21 all\'ee de Brienne, 31000 Toulouse, France.\\
\printead{e1}\\}

\affiliation{Institut Math\'ematiques de Toulouse, Universit\'e Paul Sabatier}
\address{Institut Math\'ematiques de Toulouse\\
 Universit\'e Toulouse 3 - Paul Sabatier\\118 route de Narbonne, 31400
 Toulouse, France.\\
\printead{e2}\\}

\address{Institut Math\'ematiques de Toulouse\\
Institut National des Sciences Appliqu\'ees\\ 135, avenue de Rangueil
31 077 Toulouse Cedex 4, France.\\
\printead{e3}\\}

\end{aug}
\begin{abstract}
Given an $n$-sample of random vectors $(X_i,Y_i)_{1 \leq i \leq n}$ whose joint law is unknown, the long-standing problem of supervised classification aims to \textit{optimally} predict the label $Y$ of a given a new observation $X$. In this context, the nearest neighbor rule is a popular flexible and intuitive method in non-parametric situations.
 Even if this algorithm is commonly used in the machine learning and statistics communities, less is known about its prediction ability in general finite dimensional spaces, especially when the support of the density of the observations is $\RR^d$. This paper is devoted to the study of the statistical properties of the nearest neighbor rule in various situations. In particular, attention is paid to the marginal law of $X$, as well as the smoothness and margin properties of the \textit{regression function} $\eta(X) = \mathbb{E}[Y \vert X]$. We identify two necessary and sufficient conditions to obtain uniform consistency rates of classification and to derive sharp estimates in the case of the nearest neighbor rule. Some numerical experiments are proposed at the end of the paper to help illustrate the discussion. 
\end{abstract}



\end{frontmatter}

\section{Introduction}

The supervised classification model has been at the core of numerous contributions to statistical literature in recent years. It continues to  provide interesting problems, both from the theoretical and practical point of views.  The classical task in supervised classification   is to predict a feature $Y \in \mathcal{M}$ when a variable of interest $X \in \mathbb{R}^d$ is observed, the set $\mathcal{M}$ being finite. In this paper, we focus on the binary classification problem where $\mathcal{M}=\lbrace 0,1 \rbrace$.  \\

In order to provide a prediction of the label $Y$ of $X$, it is assumed that a training set $\mathcal{S}_n = \lbrace (X_1,Y_1), \dots , (X_n,Y_n) \rbrace$ is at our disposal, where $(X_i,Y_i)$ are i.i.d. and with a common law $\mathbb{P}_{X,Y}$. This training set $\mathcal{S}_n$ makes it possible to retrieve some information on the joint law of $(X,Y)$ and to provide, depending on some technical conditions, a pertinent prediction. In particular, the regression function $\eta$ defined as:
$$ \eta(x) = \mathbb{E}[ Y \vert X=x ], \, \forall x\in \mathbb{R}^d$$
appears to be of primary interest for the statistician (see Section \ref{s:model} for a formal description of the model). Indeed, given $x\in \mathbb{R}^d$, the term $\eta(x)$ provides the probability that $Y$ is assigned the label $1$, conditionally to the event $\lbrace X=x \rbrace$. Since this function is unknown in practice, prediction rules are based on the training sample $\mathcal{S}_n$.  \\

Several algorithms have been proposed over the years but we do not intend to provide an exhaustive list of the associated papers. For an extended introduction to the supervised classification theory, we refer to \cite{survey} or \cite{DGL}. Among available classification procedures, we can, roughly speaking, divide them into (at least) three families:
\begin{itemize}
\item \textit{Approaches based on pure entropy considerations and Empirical Risk Minimization (ERM):} 
Given a classifier, the miss-classification error can be empirically estimated from the learning sample. The ERM algorithm then selects the classifier that minimizes this empirical risk among a given family
of candidates.  Several studies such as in  \cite{MamTsy}, \cite{empimini}, \cite{AudTsy}, \cite{LM_2014} now provide  an almost complete description of their statistical performance. In an almost similar context, some aggregation schemes, first proposed in Boosting procedures by \cite{Freund}, have been analyzed in depth in \cite{Lecue} and shown to be adaptive to margin and complexity.
\item \textit{Methods derived from geometric interpretation or information theory:} For example, the Support Vector Machine classifier (SVM) aims to maximize the margin of the classification  rule. It has been intensively studied in the last two decades because of its  low computational cost and excellent statistical performances (see \cite{Vapnik}, \cite{Steinwart} or \cite{Massart} among others). 
Classification and Regression Tree is another intuitive standard method that relies on a recursive dyadic partition of the state space, introduced in \cite{Breiman} and greatly improved by an averaging procedure in \cite{Amit} and \cite{BreimanRF}, which is  usually referred to as Random Forest, and later theoretically developed in \cite{Biau}. 
\item \textit{Plug-in rules:} The main idea is to mimic the Bayes optimal classifier using a plug-in rule after a preliminary estimation of the function $\eta$. We refer to \cite{GKKW2002} for a general overview, and to \cite{Bickel_Ritov2003} and  \cite{AudTsy} for some recent statistical results within this framework. 
The main motivation behind plug-in rules is to transfer properties related to the classical regression problem (estimation of  $\eta$ from the sample $\mathcal{S}_n$) to a quantitative control of the miss-classification error. 
\end{itemize}

In this general overview, the nearest neighbor rule  (see Section \ref{sec:knndef} for a complete description) belongs to the last two classes. It corresponds to a plug-in classifier with a simple geometrical interpretation. It has attracted a great deal of attention for the past few decades, from the seminal works of \cite{Fix} and \cite{Covert}. Given an integer $k$, the corresponding classifier is based on a feature average of the $k$-closest observations of $X$ in the training set $\mathcal{S}_n$. 
 We  also refer to \cite{Stone}, \cite{Gyorfi78}, \cite{Gyorfi81}, \cite{Devroye} and \cite{DGKL} for seminal contributions on this prediction rule (both for classification and regression).  
 Recently, this algorithm has received even further attention in mathematical statistics, and is still at the core of several studies:  \cite{Cerou_Guyader} examines the situation of general metric space and identifies the importance of the so-called Besicovitch assumption,  \cite{Hall} is concerned with the influence of the integer $k$ on the excess risk of the nearest neighbor rule as well as the two notions of the sample structure while \cite{Samworth} describes an improvement of the standard algorithm. 
\\

Most of the results obtained for penalized ERM, SVM or plug-in classifiers are based on complexity considerations (metric entropy or Vapnik dimension). In this paper, we mainly use the asymptotic behavior of the small ball probabilities instead (see \cite{Lian} and the references therein), which can be seen as a dual quantity of the entropy (see \cite{LiShao}). We also deal with the more intricate situation of not bounded away from zero densities (especially for non compactly supported measures).
 For this purpose, we  work with both \textit{smoothness} and \textit{minimal mass} assumptions (see Section \ref{s:rates} for more details) that will provide a pertinent estimation of the function $\eta$. In particular, it is assumed that we will be able to take advantage of some smoothness properties of the function $\eta$ in order to improve the prediction of the label $Y$. According to previous existing studies (see, \textit{e.g.}, \cite{Gyorfi81}), the associated classification rates appear to be comparable to those obtained in an ``estimation" framework and, hence, always greater than $\sqrt{n}^{-1}$. However, it has been proven in \cite{MamTsy} that \textit{fast rates} (\textit{i.e.}, faster than $\sqrt{n}^{-1}$) can be obtained up to some additional \textit{margin} assumption. It is in fact possible to take advantage of the behavior of the law of $(X,Y)$  around the boundary $\lbrace \eta =1/2 \rbrace$ in order to improve the properties of the classification process.      \\

In this paper, we investigate the nearest neighbor rule with margin assumption and marginal distribution $\mu$ for the variable $X$ that is not necessarily compactly supported or lower bounded from zero. The contributions proposed below can be broken down into three different categories. 

\paragraph{Consistency rate for bounded from below densities} Our first result concerns the optimality of the nearest neighbor classifier $\Phi_n$ in the compact case. We prove that this classification rule reaches the minimax  rate of convergence for the excess risk obtained by \cite{AudTsy} (see Theorem \ref{theo:opt} below). In particular, under some classical assumptions about the distribution $F$ of the couple $(X,Y)$ (which will be illustrated below), we show that:
$$
\sup_{F\in\mathcal{F}} \left[ \mathcal{R}(\Phi_n)- \mathcal{R}(\Phi^{*}) \right] \leq C n^{-\frac{1+\alpha}{2+d}}, 
$$
where $\alpha$ denotes the margin parameter, $d$ the dimension of the problem, $\mathcal{R}(\Phi)$ the miss-classification error of a given classifier $\Phi$ and $\Phi^*$ the Bayes classifier.\footnote{This result has also been established in the recent work of \cite{Samworth}} We obtain this result for both Poisson and Binomial sample-size  models. In particular, such a result appears to be a generalization of the ones given in \citep{Hall} that do not take the margin $\alpha$  into account in their study.

\paragraph{Consistency rate for general densities} In a second step, we investigate the behavior of the nearest neighbor classifier when the marginal density $\mu$ (w.r.t. the Lebesgue measure) of $X$ is not bounded from below on its support. Such an improvement is not of secondary importance since it corresponds to the commonly encountered situation of vanishing or non-compactly supported densities.
 To do this, we use an additional assumption on the \textit{tail} of this distribution and prove that generically:
$$
\sup_{F\in\mathcal{F}} \left[ \mathcal{R}(\Phi_n)- \mathcal{R}(\Phi^{*}) \right] \leq C n^{-\frac{1+\alpha}{2+\alpha+d}},
$$
as soon as the bandwidth $k$ involved in the classifier is allowed to depend on the spatial position of $X$.  The \textit{tail} assumption on the marginal distribution on $X$ involved in this result, will describe the behavior of the density $\mu$ near the set $\{\mu = 0\}$.

\paragraph{Lower bounds} Finally, we derive some lower bounds for the supervised classification problem, which extends the results obtained in \cite{AudTsy} in a slightly different context. We prove that our Tail Assumption is unavoidable  to ensure uniform consistency rates for classification in a non-compact case, regardless of dimension $d$. We then see how these upper and lower bounds are linked. In particular, we show that a very unfavorable situation of classification occurs when the regression function $\eta$ oscillates in the tail of the distribution $\mu$, \textit{i.e.}, we establish that it is even impossible in these situations to obtain uniform consistency rates and thus elucidate two open questions in \cite{Cannings}.\\

\quad The paper is organized as follows. In Section \ref{s:model}, we precisely describe the statistical setting related to the classification problem. Some attention is paid to the nearest neighbor rule. Section \ref{s:compact} is devoted to the bounded from below case where we prove that the nearest neighbor classifier reaches the minimax rate of convergence for the excess risk under mild assumptions. We then extend our study to the general (typically non-compact) case in Section \ref{s:pascompact}.	This section is supplemented with some supporting numerical results and a glossary of typical situations of location models. We conclude with a discussion of our results, and potential problems.
Proofs and technical results are included in Appendix \ref{s:proofs}. The paper is completed by an adaptation to the smooth discriminant analysis model in Appendix \ref{s:sda} (see, \textit{e.g.} \cite{MamTsy} or \cite{Hall} for another comparison between the so-called Poisson and Binomial models). In particular, although the variables of interest are strongly dependent in this case, we derive (using a Poissonization argument) results similar to those obtained in the classical binary classification model. \\

We use the following notations throughout the paper. $\mathbb{P}_{X,Y}$ denotes the distribution of the couple $(X,Y)$ and $\mathbb{P}_X$ the marginal distribution of $X$, which will be assumed to admit a density $\mu$ with respect to the Lebesgue measure. Similarly, we set $\mathbb{P}_{\otimes^n}= \prod_{i=1}^n \mathbb{P}_{(X_i,Y_i)}$ and $\mathbb{P} = \mathbb{P}_{(X,Y)} \times \mathbb{P}_{\otimes^n}$.  In the same spirit, $\mathbb{E}[.]$, $\mathbb{E}_X[.]$ and $\mathbb{E}_{\otimes^n}[.]$ will hereafter correspond to the expectations w.r.t. the measures $\mathbb{P}$, $\mathbb{P}_X$ and $\mathbb{P}_{\otimes^n}$, respectively.  Finally, given two real sequences $(a_n)_{n\in \mathbb{N}}$ and $(b_n)_{n\in \mathbb{N}}$, we write $a_n \lesssim b_n$ (resp. $a_n \sim b_n$) 
is a real constant $C\geq 1$ exists such that $a_n \leq C b_n$ 
(resp. $\frac{b_n}{C} \leq a_n \leq C b_n$) for all $n\in \mathbb{N}$.


\section{Statistical setting and nearest neighbor classifier}
\label{s:model}

\subsection{Statistical Classification problem}
\label{s:setting}

In this paper, we study the classical binary supervised classification model (see, \textit{e.g.}, \cite{DGL} for a complete introduction). An i.i.d. sample $\mathcal{S}_n:=(X_i,Y_i)_{i=1...n}\in\Omega \times \{0,1\}$, whose distribution is $\mathbb{P}_{X,Y}$ and where $\Omega = \mathrm{Supp}(\mu)$ is an open set of $\RR^d$, is at our disposal.  Given a new incoming observation $X$, our goal is to predict its corresponding label $Y$. To do this, we use a classifier that provides a decision rule for this problem. Formally, a classifier is a measurable mapping  $\Phi$ from $\RR^{d}$ to $\{0,1\}$. Given a classifier $\Phi$, its corresponding miss-classification error is then defined as:
$$
\mathcal{R}(\Phi)=\PP\left(\Phi(X)\neq Y\right).
$$
In practice, the most interesting classifiers are those associated with the smallest possible error. In this context it is well known (see, \textit{e.g.}, \cite{survey}) that the Bayes classifier $\Phi^{*}$ defined as:
\begin{equation}
\Phi^{*}(X)=\mathbf{1}_{\left\lbrace \eta(X)>\frac12 \right\rbrace},\ \mathrm{where\ } \eta(x):=\EE\left[Y|X=x\right] \ \forall x\in \Omega,
\label{eq:Bayes_classic}
\end{equation}
minimizes the miss-classification error, \textit{i.e.},
$$
\mathcal{R}(\Phi^{*})\leq \mathcal{R}(\Phi),\ \forall \Phi:\RR^{d}\longrightarrow\{0,1\}.
$$
The classifier $\Phi^\star$ provides the best decision rule in the sense that it leads to the lowest possible miss-classification error.  Unfortunately, $\Phi^{*}$ is not available since the regression function $\eta$ explicitly depends on the underlying distribution of $(X,Y)$. In some sense, the Bayes classifier can be considered as an \textit{oracle} that provides a benchmark error. Hence, the main challenge in this supervised classification setting is to construct a classifier $\Phi$ whose miss-classification error will be as close as possible to the smallest possible one. In particular, the excess risk (also referred to as the \textit{regret}) defined as 
$$
\mathcal{R}(\Phi)- \mathcal{R}(\Phi^{*}),
$$
appears to be of primary importance. We are interested here in the statistical properties of the nearest neighbor  classifier (see Section \ref{sec:knndef} below for more details) based on the sample $\mathcal{S}_n$. In particular, we investigate the asymptotic properties of the excess risk through the \textbf{minimax} paradigm. Given a set $\mathcal{F}$ of possible distributions $F$ for $(X,Y)$, the minimax risk is defined as:
$$ \delta_n(\mathcal{F}) := \inf_{ \Phi}\sup_{F\in\mathcal{F}} \left[ \mathcal{R}(\Phi)- \mathcal{R}(\Phi^{*})\right],$$
where the infimum in the above formula is taken over all $\mathcal{S}_n$ measurable classifiers. A classifier $\Phi_n$ is then said to be minimax over the set $\mathcal{F}$ if:
$$
\sup_{F\in\mathcal{F}} \left[ \mathcal{R}(\Phi_{n})- \mathcal{R}(\Phi^{*}) \right] \leq C \delta_n(\mathcal{F}),
$$
for some constant $C$. The considered set $\mathcal{F}$ will be detailed  later on and will depend on the behavior of $(\mu,\eta)$ over $\RR^{d}$ through some smoothness, margin and minimal mass hypotheses.

\subsection{The nearest neighbor rule \label{sec:knndef}}
In this paper, we focus on the nearest neighbor classifier, which is perhaps one of the most widespread and simplest classification procedures.  Suppose that the  state space  is $(\RR^d,\|.\|)$ where $\|.\|$ is a reference distance. Given any sample $\mathcal{S}_n$ and for any $x\in \RR^{d}$, we first build the reordered   sample $\left(X_{(j)}(x),Y_{(j)}(x)\right)_{1\leq j\leq n}$ with respect to the distances $\|X_{i}-x\|$, namely:
$$
\|X_{(1)}(x) - x\| \leq \|X_{(2)}(x) - x\| \leq  \ldots \leq \|X_{(n)}(x) - x\|.
$$
In this context $X_{(m)}(x)$ is the $m$-nearest neighbor of $x$ w.r.t. the distance $\| . \|$ and  $Y_{(m)}(x)$ its corresponding label. Given  any integer $k$ in $\NN$, the principle of the nearest neighbor algorithm is to construct a decision rule based on the $k$-nearest neighbor of the input $X$: the $\mathcal{S}_n$-measurable classifier $\Phi_{n,k}$ is:
\begin{equation}\label{eq:kNN}
\Phi_{n,k}(X)=\left\{\begin{matrix}
1& \mathrm{if} & \displaystyle\frac1k\sum_{j=1}^{k}Y_{(j)}(X)>\frac12,\\
0& \mathrm{otherwise}.
\end{matrix}\right.
\end{equation}
For all $x\in \Omega$, the term $ \frac1k\sum_{j=1}^{k}Y_{(j)}(x)$ appears to be an estimator of the regression function $\eta(x)$. In particular, we can write the classifier $\Phi_{n,k}$ as 
\begin{equation}
\Phi_{n,k}(X) = \mathbf{1}_{\left\lbrace \hat \eta_n(X)>1/2 \right\rbrace} \quad \quad \mathrm{where} \quad \quad \hat\eta_n(x) = \displaystyle\frac1k\sum_{j=1}^{k}Y_{(j)}(x) \quad \forall x\in \Omega.
\label{eq:def_eta_n}
\end{equation}
Hence, the nearest neighbor procedure can be considered as a plug-in classifier, \textit{i.e.}, a preliminary estimator of the regression function $\eta$ is plugged in our decision rule. It is worth noting that the integer $k$ is a \textit{regularization parameter}. Indeed, if $k$ is too small, the classifier $\Phi_{n,k}$ will only use a small amount of the neighbors of $X$, leading to a large variance during the classification process. On the other hand, large values of $k$ will introduce some bias into the decision rule since we use observations that may be far away from the input $X$. In other words,  the statistical performances of $\Phi_{n,k}$  will depend on a careful choice of the integer $k$. In particular, the number of neighbors $k=k_n$ considered should  carefully grow to $+\infty$ with respect to $n$. 
\\

For this purpose, we introduce some baselines assumptions into the following section that will make it possible to characterize an optimal value for this regularization parameter. 


\subsection{Baseline assumptions}
\label{s:rates}

It is well known that no reliable prediction can be made in a distribution-free setting (see \cite{DGL}). We restrict the class of possible distributions of $(X,Y)$ below.\\

Since the nearest neighbor rule is a plug-in classification rule, we expect to take advantage of some smoothness properties of $\eta$ in order to improve the classification process. In fact, when $\eta$ is smooth, the respective values of $\eta(x_1)$ and $\eta(x_2)$ are comparable for close enough $x_1,x_2$. In other words, we can infer the sign of $\eta(x)-\frac{1}{2}$ from those of the neighbors of $x$.  

\paragraph{Assumption \textbf{A1}}(Smoothness) \textit{The regression function $\eta$ belongs to the H\"older class of parameter $1$ with a \textit{radius} $L$, which is denoted $\mathcal{C}^{1,0}(\Omega,L)$ and corresponds to the set of functions such that
 $$
 \forall (x_1,x_2) \in \Omega^2 \qquad |\eta(x_1) - \eta(x_2)| \leq L |x_1-x_2|.
 $$
}

\begin{rem} 
\textrm{It would be tempting to consider some more general smoothness classes for the regression function $\eta$. Nevertheless, the standard nearest neighbor algorithm does not make it possible to use smoothness indexes greater than $1$. An alternative procedure has been proposed in \cite{Samworth}: the idea is then to balance the $(Y_{(j)})_{j=1..k}$ with  a suitable monotonous weighting sequence. However, this modification complicates the statistical analysis and may alter the ideas developed below. We therefore chose to fix the smoothness of $\eta$ to $1$ (\textit{i.e.} restrict our study to $\mathcal{C}^{1,0}(\Omega,L)$).}
\end{rem} 

Our second assumption  was introduced by \cite{Tsybakov_tout_seul} in the binary supervised classification model (see \cite{MamTsy} in a \textit{smooth discriminant analysis} setting).

\paragraph{Assumption \textbf{A2}} (Margin assumption) \textit{For any $\alpha>0$, a constant $C>0$ exists such that:
$$ 
\mathbb{P}_X\left(0<\left|\eta(X)-\frac12\right|<\epsilon\right)\leq C\epsilon^{\alpha}, \quad  \forall \epsilon >0.
$$
In such a case, we write $(\mu,\eta) \in \mathbf{M}_{\alpha}$}.\\

The Bayes classifier depends on the sign of  $\eta(X)-1/2$. Intuitively, it would be easier to mimic the behavior of this classifier when the mass around the set $\lbrace \eta = 1/2 \rbrace$ is small. On the other hand, the decision process may be more complicated when $\eta(X)$ is close to $1/2$ with a large probability. Quantifying this closeness is the purpose of this margin assumption. \\

For the sake of convenience, we use the set $\mathcal{F}_{L,\alpha}$ throughout the paper, which contains distributions that satisfy both Assumptions \textbf{A1} and \textbf{A2}, namely:
\begin{multline*}
\mathcal{F}_{L,\alpha} : = \left\{ \mathbb{P}_{(X,Y)} \, : \,  \mathbb{P}_{X}(dx) =  \mu(x)dx \, \text{and} \, \mathcal{L}(Y \vert X) \sim \mathcal{B}(\eta(X)) \right. \\
\left. \text{with} \, \eta \in \mathcal{C}^{1,0}(\Omega,L) \, \text{and} \,  (\mu,\eta) \in \mathbf{M}_{\alpha} \right\}
\end{multline*}

We now turn to our last assumption that involves  the marginal distribution of the variable $X$.

\subsection{Minimal Mass Assumption}

In the sequel, this type of hypothesis will play a very important role. 

\paragraph{Assumption \textbf{A3}}(Strong Minimal Mass Assumption) \textit{There exists $\kappa>0$ such that the marginal density $\mu$ of $X$ satisfies $\mu \in \blabla$ where}

\begin{multline*}
\blabla := \bigg\{  \mathbb{P}_{X} \, : \,  \mathbb{P}_{X}(dx) =  \mu(x)dx \, \vert \\
 \,  \exists \delta_0> 0,  \, \forall \delta \leq \delta_0, \,
  \forall x \in \Omega: \mathbb{P}_X(X \in B(x,\delta)) \geq \kappa \mu(x) \delta^d  \bigg\}.
\end{multline*}

This assumption guarantees that $\mathbb{P}_X$ possesses a minimal amount of mass on each ball $B(x,\delta)$, this lower bound being balanced by the level of the density on $x$.  In some sense, distributions in $\blabla$ will make it possible to obtain reliable predictions of the regression function $\eta$ according to its Lipschitz property. The Strong Minimal Mass Assumption \textbf{A3} may be seen as a refinement of the so-called Besicovitch assumption that is quite popular in the statistical literature (see, \textit{e.g.}, \cite{Dev81} for a version of the Besicovitch assumption used for pointwise consistency or \cite{Cerou_Guyader} for a general discussion on this hypothesis in finite or infinite dimension). It is worth pointing out that the Besicovitch assumption introduced in \cite{Cerou_Guyader} states that $\eta$ satisfies the following $\mu$-continuity property:

\begin{equation}
\forall \epsilon >0 \quad \lim_{\delta \to 0} \mathbb{P}_X \left\{ x : \frac{1}{\mu(B(x,\delta))} \int_{B(x,\delta)} |\eta(z)-\eta(x)| d \mu(z) > \epsilon \right\}=0
\label{eq:Besi_Guyader}
\end{equation}
In our setting, since $\eta$ is $L$-Lipschitz (Assumption \textbf{A1}), we can check that for all $x\in \Omega$
$$
\int_{B(x,\delta)} |\eta(z)-\eta(x)| \mu(z)dz  \leq L \int_{B(x,\delta)} |x-z| \mu(z)dz \leq L\delta \mu(B(x,\delta)),$$
which implies that the right hand side of (\ref{eq:Besi_Guyader}) vanishes as soon as $\delta \leq \epsilon/L$. We will see that Assumption \textbf{A3} is necessary to obtain quantitative estimates for any finite dimensional classification problem in a general setting.


In a slightly different framework, our Assumption \textbf{A3} is similar to the \textit{Strong Density Assumption}
 used in the paper of \cite{AudTsy} when the density $\mu$ is lower bounded on its (compact) support, which is assumed to possess some geometrical properties ($(c_0,r_0)$ regularity). This setting is at the core of the study presented in Section \ref{s:compact} below. Assumption \textbf{A3} also recalls the notion of standard sets used in \cite{casal} for the estimation of compact support sets. More generally, the following examples present some standard distributions that satisfy Assumption \textbf{A3}.\\

\begin{ex}
$ $ 
\begin{itemize}
\item  In $\mathbb{R}^d$, it is not difficult to check that Gaussian measures with non-degenerated covariance matrices satisfy 
$\blabla$. As a simple example, consider a standard Gaussian law $\mu \sim \mathcal{N}(0,1)$. For any $x \in \RR$ and $\delta>0$, if $x$ belongs to a compact set $K$, then a constant $C_K$ exists such that
$(2 \pi)^{-1/2} \int_{x-\delta}^{x+\delta} e^{-t^2/2} dt \geq C_K e^{-x^2/2} \delta$. Now, if $x\longrightarrow + \infty$, we can check that:
$$
(2 \pi)^{-1/2}  \int_{x-\delta}^{x+\delta} (2 \pi)^{-1/2}  e^{-t^2/2} dt  \sim (2 \pi)^{-1/2}   e^{-x^2/2}  \left[ \frac{e^{x \delta }}{x-\delta} - \frac{e^{-x \delta }}{x+\delta}   \right] e^{-\delta^2/2}.
$$
The bracket above is always greater than $\delta$ when $(x \delta)^{-1} = O(1)$. Now, if $\delta = o(1/x)$, a simple Taylor expansion yields 
$$(2 \pi)^{-1/2}  \int_{x-\delta}^{x+\delta} (2 \pi)^{-1/2}  e^{-t^2/2} dt  \sim \mu(x) \frac{1+2 x \delta}{x} \gtrsim
\mu(x) \delta. $$
\item The same computations are still possible for symmetric Laplace distributions ($e^{t} \int_{t-\delta}^{t+\delta} e^{-x} dx =[ e^{\delta} - e^{-\delta}] \sim 2 \delta$ when $\delta$ is small. Thus, any Laplace distribution belongs to $\blabla$. In a same way, when $\mu$ is a standard Cauchy distribution, we can check that:
\begin{eqnarray*}
\int_{x-\delta}^{x+\delta} \frac{dt}{1+t^2}& =& \frac{1}{1+x^2} \int_{\delta}^{\delta} \frac{1}{1+h \frac{2x+h}{1+x^2}} dh \\
& \sim& \frac{1}{1+x^2} \left[ 2 \delta - \frac{2}{3} \frac{\delta^3}{1+x^2} + +8\frac{\delta^3x^2}{(1+x^2)^2}o(\delta^3)\right] \\
& \gtrsim &\frac{\delta}{1+x^2}
\end{eqnarray*}
\end{itemize}
\end{ex}

Typically, distributions that do not satisfy the Strong Minimal Assumption (A3) possess some important oscillations in their tails (when the density $\mu$ is close to $0$). In such a setting, the alternative set $\bla$, defined as follows, may be considered:
\begin{multline*}
\bla := \bigg\{  \mathbb{P}_{X} \, : \,  \mathbb{P}_{X} (dx) = \mu(x) dx \vert \, \exists (\rho,C) \in ]0;+\infty[^2  \\ \exists \delta_0> 0,  \, \forall \delta \leq \delta_0:
  \forall x \in \Omega: \mu(x) \geq e^{-C \delta^{-\rho}} \, \Longrightarrow  \,  \mathbb{P}_{X} (B(x,\delta)) \geq \kappa \mu(x) \delta^d \bigg\}.
\end{multline*}
The interest of the weaker $\bla$ compared to $\blabla$ is that the statistical abilities of the nearest neighbor rule are still the same with $\blabla$ or $\bla$.
Moreover, an analytic criterion that ensures  $\bla$ can be found (see Proposition \ref{prop:sufficient}. This is not the case for the \textit{uniform} assumption $\blabla$ (it is indeed more difficult to ensure the lower bound on the global set $\Omega$).


\vspace{0.5cm} 
Although all the subsequent results may be established for a weaker version of the minimal mass assumption (based on the set $\bla$), we will restrict ourselves to its strong formulation (Assumption \textbf{A3}). In Section \ref{s:compact}, we prove that the nearest neighbor rule is optimal in the minimax sense provided that the margin and smoothness assumptions hold, with a marginal density of the variable $X$ bounded away from $0$ and a suitable choice of $k$. In Section \ref{s:pascompact}, we will see that $\blabla$ is not yet sufficient to derive consistent classifiers for non compactly supported densities, and a last additional hypothesis is needed.


\section{Bounded away from zero densities}
\label{s:compact}

\subsection{Minimax consistency of the nearest neighbor rule}

In this section, we are interested in the special case of a marginal density $\mu$ bounded from below by a strictly positive constant $\mu_{-}$. In this context, we can state an upper bound on the consistency rate of the nearest neighbor rule. 

\begin{theo}
\label{theo:Main1} 
Assume that Assumptions \textbf{A1-A3} hold. The nearest neighbor classifier $\Phi_{n,k_n}$ with $k_n =  \lfloor n^{\frac{2}{2+d}}\rfloor $ satisfies
$$
\sup_{\mathbb{P}_{X,Y} \in\mathcal{F}_{L,\alpha} \cap\blabla^{\mu_-}} \left[ \mathcal{R}(\Phi_{n,k_n})- \mathcal{R}(\Phi^{*}) \right] \lesssim  n^{-\frac{1+\alpha}{2+d}},
$$
where $\blabla^{\mu_-}$ denotes the subset of densities of  $\blabla$ that are bounded from below by $\mu_{-}$.
\end{theo}

Theorem \ref{theo:Main1} establishes a consistency rate of the nearest neighbor rule over $\mathcal{F}_{L,\alpha} \cap\blabla^{\mu_-}$. A detailed proof of is presented in Section \ref{preuve_dim_finie}. Implicitly, we restrict our analysis to compactly supported observations, this assumption being at the core of several statistical analyses (see, \textit{e.g.}, \cite{GKKW2002}, \cite{survey}, \cite{MamTsy} or \cite{Hall} among others). It is worth pointing out that this setting falls into the framework considered in \cite{AudTsy}. 


\begin{defi}[Strong Density Assumption (SDA), \cite{AudTsy}]
The marginal distribution of the variable $X$ satisfies the Strong Density Assumption if
\begin{itemize}
\item  it admits a density $\mu$ w.r.t. the Lebesgue measure of $\RR^d$,
\item  the density $\mu$ satisfies:
$$
 \mu_- \leq \mu(x) \leq \mu_+,\quad \forall x \in \mathrm{Supp}(\mu) 
$$
for some constants $(\mu_-,\mu_+) \in ]0,+\infty[^2$. 
\item The support of $\mu$ is \textit{$(c_0,r_{0})$-regular}, namely:
$$
\lambda \left[ \mathrm{Supp}(\mu) \cap B(x,r) \right] \geq c_0 \lambda [B(x,r)], \forall r\leq r_{0},
$$ 
for some positive constants $c_0$ and $r_0$.  
\end{itemize}
\end{defi}

As soon as the marginal density is bounded from below by a strictly positive constant, then both SDA  and Strong Minimal Mass Assumption (\textbf{A3}) are equivalent, as stated in the following proposition.

\begin{prop} 
\label{prop:equi} For bounded away from zero density, the SDA is equivalent to the Strong Minimal Mass Assumption.
\end{prop}

\begin{proof} As soon as the support of $\mu$ is $(c_0,r_{0})$-regular and the density is lower bounded by $\mu_{-}>0$, then SDA implies a minimal mass type assumption since $\forall \delta \leq r_0 $:
$$
  \mathbb{P}_{X} (B(x,\delta)) = \int_{B(x,\delta)} \mu(z) dz \geq \mu_- \times  \lambda[B(x,\delta) \cap  \mathrm{Supp}(\mu) ] \geq c_0 \gamma_d \mu_{-} \delta^d.
$$

Conversely, we can also check the fact that the Strong Minimal Mass Assumption (A3) implies the SDA (including the $(c_0,r_{0})$-regularity of $\mu$). Indeed, since for any $x$ and $\delta\leq \delta_0$:
$$
1 \geq \int_{B(x,\delta)} \mu(x) dx \geq C \mu(x) \delta^d, 
$$
then the density $\mu$ is upper bounded and we obtain that:
$$
\int_{B(x,\delta)} \mu(x) dx \leq \|\mu\|_{\infty}\lambda \left[ \mathrm{Supp}(\mu) \cap B(x,r) \right].
$$
We therefore obtain:
$$
\lambda \left[ \mathrm{Supp}(\mu) \cap B(x,r) \right] \geq C \frac{\mu(x)}{\|\mu\|_{\infty}} \delta^d \geq C \frac{\mu_-}{\|\mu\|_{\infty}} \delta^d.
$$
This concludes the proof of this proposition. \end{proof}

It is possible to link the constants $(c_0,r_0)$  involved in SDA with  $\kappa$ involved in $\blabla^{\mu_{-}}$, but we have omitted their relationships here for the sake of simplicity.
Minimax rates of excess risk under the SDA are established in \cite{AudTsy}. A consequence of Proposition \ref{prop:equi} is that the same lower bound is still valid with $\blabla^{\mu_{-}}$.

\begin{theo}[Theorem 3.3, \cite{AudTsy}]\label{theo:opt}
Assume that Assumptions \textbf{A1-A3} hold and a $\mu_->0$ exists such that $\mu(x)>\mu_-$ for all $x\in \Omega$. Then,  the minimax classification rate is lower bounded as follows:
$$
\inf_{ \Phi}\sup_{\mathbb{P}_{X,Y}\in\mathcal{F}_{L,\alpha} \cap \blabla^{\mu_{-}}} \left[ \mathcal{R}(\Phi)- \mathcal{R}(\Phi^{*}) \right] \gtrsim n^{-\frac{1+\alpha}{2+d}}.
$$
\end{theo}

Thanks to the previous lower bound, we can conclude that the nearest neighbor rule achieve the minimax rate of convergence in the particular case where the density $\mu$ is lower bounded on its (compact) support. As already discussed in \cite{MamTsy} or \cite{AudTsy}, the higher the margin index $\alpha$ is, the smaller the excess risk will be. On the other hand, the performance deteriorates as the dimension of the considered problem increases. This corresponds to the classical curse of the dimensionality. The lower bound obtained by \cite{AudTsy} is based  on an adaptation of standard tools from nonparametric statistics (Assouad's Lemma). This proof is of primary importance for next lower bound results. It is recalled in Section \ref{s:proofs} for the sake of convenience.

\subsection{The Smooth discriminant analysis model (Binomial sample-size)}

While the supervised classification model (also referred to as the Poisson sample-size model) has been intensively studied in the last decades, the smooth discriminant analysis model has been considered as an alternative approach. This model is presented in \cite{MamTsy} and is referred to as a binomial model in \cite{Hall}. It assumes that we have two independent samples $\mathcal{S}_1=(X_1,\dots,X_n)$ and $\mathcal{S}_2=(\tilde X_1,\dots,\tilde X_n)$ of i.i.d. random variables at our disposal, with densities $f$ and $g$ respectively. Given a new incoming observation, the goal is then to predict its corresponding label, namely to determine whether $X$ comes from the density $f$ or $g$. \\

\qquad 
In the classification setting, the positions are drawn according to $\mu$ and the labels are then sampled using $\mathcal{B}(\eta(X))$, which makes the values of the labels $(Y_{(i)})_{1 \leq i \leq n}$ completely independent each other, conditionally to their positions $(X_{(i)})_{1 \leq i \leq n}$. This key observation is no longer true in the smooth discriminant analysis: conditionally to ordered spatial inputs  induced in the nearest neighbor rule, the random variables $(Y_{(1)}, \ldots, Y_{(k_n)})$ \textbf{are not independent}. This significantly complicates the analysis of the nearest neighbor rule and is a major difference with the standard classification task. \\

\qquad 
We briefly provide our main result on the nearest neighbor rule with the smooth discriminant analysis below. More complete details can be found in Appendix  \ref{s:sda}.

\begin{theo}
\label{theo:sda} 
The nearest neighbor classifier $\Phi_{n,k_n}$ with $k_n =  \lfloor n^{\frac{2}{2+d}}\rfloor $ satisfies
$$
\sup_{\mathbb{P}_{X,Y} \in\mathcal{F}_{L,\alpha} \cap\blabla^{\mu_-}} \left[ \mathcal{R}^{Binom}(\Phi_{n,k_n})- \mathcal{R}^{Binom}(\Phi^{*}) \right] \lesssim  \log(n) n^{-\frac{1+\alpha}{2+d}},
$$
where  $\mathcal{R}^{Binom}$ denotes the risk in the smooth discriminant analysis setting.
\end{theo}

\qquad
To the best of our knowledge,  the performance of the nearest neighbor classifier in the binomial sample-size model has only been studied in \cite{Hall}. In their paper, the difference between the Poisson and the binomial model is studied through Reny's representation of order statistics. In contrast, we directly compute an upper bound of the binomial model. Our main argument relies on a Poissonization of the sample size (see, \textit{e.g.}, \cite{Kac}). Even if it is a standard alternative to cope with dependencies in probability, such a method has not yet been applied for smooth discriminant analysis.

\qquad
Regarding the obtained consistency rates now, our result misses a log term in the smooth discriminant analysis setting. In \cite{Hall}, the authors show that the difference of the excess risk between the classification and the smooth discriminant analysis is on the order of $o\left(k^{-1}+\left(\frac{k}{n}\right)^{4/d}\right)$ for \textit{twice} differentiable functions $\eta$ (instead of only the Lipschitz situation in our case) and their resulting rate is $n^{-2/(4+d)}$ for the optimal choice $k_n = n^{4/(4+d)}$. Following their argument with a \textit{Lipschitz} regression function $\eta$, their excess risk becomes $n^{-1/(2+d)}$ for the binomial model. Hence, for a margin $\alpha=0$, our result in Theorem \ref{theo:sda} is weaker than the one in \cite{Hall} (because of our log term). This is not yet the case as soon as the margin $\alpha>0$ since the result of \cite{Hall} does not take this parameter, which may be central to obtain fast rates, into account. Moreover, the approach of \cite{Hall} does not seem to simply manage the margin information of the classification.

\qquad
Finally, our Poissonization method also applies for general densities that are not necessarily bounded from below (see Appendix \ref{s:sda}). This is a major difference with the results of \cite{Hall} that are valid with a compactly supported and bounded away from   zero density $\mu$.


\section{General finite dimensional case}
\label{s:pascompact}

\subsection{The Tail Assumption}
Results of the previous section are designed for the problem of supervised binary classification with compactly supported inputs and lower bounded densities. Such an assumption is an important prior on the problem that may be improper in several practical settings. Various situations involve Gaussian, Laplace, Cauchy or Pareto distributions on the observations, and both  the compactness  and the boundedness away from zero assumptions may seem to be very unrealistic. This is even more problematic when dealing with functional classification with a Gaussian White Noise model (GWN). In such a case, observations are described through an infinite sequence of Gaussian random variables and the SDA or $\blabla^{\mu_{-}}$ are far from being well-tailored for this situation (see \cite{Lian} for a discussion and further references).

\vspace{1em}
\qquad
This section is dedicated to a more general case of binary supervised classification problems where the marginal density $\mu$ of $X$ is no longer assumed to be lower bounded on its support. The main problem related to such a setting is that we have to predict labels in places where few (or even no) observations are available in the training set. In order to address this problem, we take the following assumption. 

\paragraph{Assumption \textbf{A4}}(Tail Assumption) \textit{A
  function $\psi$ that satisfies $\psi(\epsilon) \rightarrow 0$ as $\epsilon \rightarrow 0$ and that increases in a neighborhood  of 0 exists such that
$$ \mathbb{P}_{(X,Y)} \in
\PTf := \left\{ \mathbb{P}_{X} \, : \, \exists\,  \epsilon_0 \in \RR_+^* :  \, \forall \epsilon < \epsilon_0, \,  \mathbb{P}_X \left( \{\mu<\epsilon\} \right) \leq \psi(\epsilon)  \right\},
$$
where $\PT$ corresponds to the particular case where $\psi = Id$. \\}


The aim of this Tail Assumption is to ensure that the set where $\mu$ is small has a small mass. We use the notation $\mathcal{T}$ because of the interpretation on the \textit{tail} of $\mu$, but  $\PTf$ is not just an assumption on the tail of the $\mu$. It is, in fact, an assumption on the behavior of $\mu$ near the set $\left\{ \mu = 0 \right\}$. We provide some examples of marginal distribution below that satisfy this tail requirement. In Section \ref{sec:noncons} below, we prove that the Tail Assumption (\textbf{A4}) is unavoidable in this setting. In Section \ref{sec:minimax8}, we investigate the performances of the nearest neighbor rule in this setting. 
\vspace{0.5cm}

\begin{ex} Following are several families of densities in $\PTf$.

\begin{itemize}
\item Laplace distributions obviously satisfy $\PT$, and a straightforward  integration by parts shows that Gamma distributions
$\Gamma(k,\theta)$ satisfy $\PTf$ with $\psi(\epsilon) = \epsilon \log(\epsilon^{-1})^{k-1}$ (the term around $x=0$ is on the order of $\epsilon^{k/(k-1)}$ and thus negligible compared to the term around $+ \infty$).
\item An immediate computation shows that the family of Pareto distributions of parameters $(x_0,k)$ satisfies $\PTf$ where $\psi(\epsilon) = \epsilon^{k/(k+1)}$, regardless of the value of $x_0$.
\item The family of Cauchy distributions satisfies $\PTf$ with $\psi(\epsilon) = \sqrt{\epsilon}$.
 \item Univariate Gaussian laws $\gamma_{m,\sigma^2}$ with mean $m$ and variance $\sigma^2$ satisfy
$$
\gamma_{m,\sigma^2}(x) \leq \epsilon \Longleftrightarrow |x-m| \geq t_{\sigma,\epsilon} := \sqrt{2} \sigma \sqrt{ \log\left(\frac{1}{\epsilon}\right) + \log(\frac{1}{\sigma \sqrt{2 \pi}} )},
$$
and a standard result on the size of Gaussian tails (see \cite{Cox}) yields
$$
\gamma_{m,\sigma^2}\left( \gamma_{m,\sigma^2} \leq \epsilon \right) = \frac{\epsilon}{t_{\sigma,\epsilon}} \left[1-\frac{1}{t_{\sigma,\epsilon}^{2}}+\frac{1.3}{t_{\sigma,\epsilon}^4} \ldots \right] \lesssim \frac{\epsilon }{\sqrt{\log \left(\frac{1}{\epsilon}\right)}}.
$$
Hence, univariate Gaussian laws satisfy $\PTf$ with $\psi( \epsilon) = \epsilon \log(\epsilon^{-1})^{-1/2}$.
\item If  $\mathbf{m}$ is any real vector of $\RR^d$ and $\Sigma^2$ a covariance matrix whose spectrum is $\lambda_1 \geq \ldots \lambda_d \geq 0$:
$$
\gamma_{\mathbf{m},\Sigma^2}\left( \gamma_{\mathbf{m},\Sigma^2} \leq \epsilon \right) = \gamma_{\mathbf{0},\Sigma^2}\left( \gamma_{\mathbf{0},\Sigma^2} \leq \epsilon \right) \lesssim \gamma_{\mathbf{0},\Sigma^2} \left( \|X\| \geq \sqrt{2 \lambda_1 \log\left( \frac{1}{\epsilon}\right)} \right).
$$
Careful inspection of Theorem 1 of  \cite{Husler} now yields
$$
 \gamma_{\mathbf{0},\Sigma^2} \left( \|X\| \geq \sqrt{2 \lambda_1 \log\left( \frac{1}{\epsilon}\right)} \right) \sim  C_{\Sigma^2}
 \log\left( \frac{1}{\epsilon}\right)^{r/2-1} \epsilon,
$$
where $C_{\Sigma^2}$ is a constant  that only depends on the spectrum of $\Sigma^2$ and $r$ is the multiplicity of the eigenvalue $\lambda_1$. In particular, $\gamma_{\mathbf{m},\Sigma^2}$ satisfy $\PTf$ where $\psi(\epsilon) = C_{\Sigma^2} \epsilon \log(\epsilon^{-1})^{r/2-1}$.

\end{itemize}
\end{ex}

\subsection{Non-consistency results}\label{sec:noncons}

We first justify the introduction of the sets $\blabla$ and $\PTf$ and discuss their influences regarding uniform lower bounds and even consistency of any estimator. To do this, we first state that the Minimal Mass Assumption (\textbf{A3}) is necessary to obtain \textit{uniformly} consistent classification rules. Second, we  assert that the Tail Assumption (\textbf{A4}) is also unavoidable.

\begin{theo}\label{theo:NON1} Assume that the law $\mathbb{P}_{X,Y}$ belongs to $\mathcal{F}_{L,\alpha}$, then:

\begin{itemize}
\item[i)]
No classification rule can be universally consistent {\em if Assumptions \textbf{A1-A3} hold and not \textbf{A4}}. For any discrimination rule $\Phi_n$ and for any $\epsilon< 4^{-\alpha}$, a distribution $\mathbb{P}_{(X,Y)}$ in $\mathcal{F}_{L,\alpha} \cap  \blabla$ exists such that:
$$
 \mathcal{R}(\Phi_n)- \mathcal{R}(\Phi^{*}) \geq  \epsilon.
$$
\item[ii)]
No classification rule can be universally consistent {\em if Assumption \textbf{A1, A2, A4} hold and not \textbf{A3}}. For any discrimination rule $\Phi_n$ and for any  $\epsilon< 4^{-\alpha}$, a distribution $\mathbb{P}_{(X,Y)}$ in $\mathcal{F}_{L,\alpha} \cap \PT $ exists such that:
$$
 \mathcal{R}(\Phi_n)- \mathcal{R}(\Phi^{*}) \geq  \epsilon.
$$
\end{itemize}
\end{theo}

\qquad The first result $i)$ asserts that  even if the Minimal Mass Assumption \textbf{A3} holds for the underlying density on $X$, it is not possible to expect a uniform consistency result over  the entire class of non-compactly considered densities. In some sense, the support of the variable $X$ seems to be too large to obtain reliable predictions with any classifiers without additional assumptions. As discussed above, the Tail Assumption \textbf{A4} may make it possible to counterbalance this curse of support effect (see next section). Such statistical damage has also been observed for the estimation of densities that are supported on the real line instead of being compactly supported, even though such dramatic consequences are not shown here. We refer to \cite{curse} and the references therein for a more detailed description.

\vspace{0.5em}
\qquad  The second result $ii)$ states that the Strong Minimal Mass Assumption \textbf{A3} cannot be skipped for uniform consistency rates and no compactly supported densities. This is in line with the former studies of \cite{Gyorfi78} and \cite{DGKL}. In particular, Lemma 2.2 of \cite{DGKL} takes advantage of some of the positive consequences of this type of assumption. Our proof relies on the construction of a sample size dependent law on $(X,Y)$ that violates our Minimal Mass Assumption \textbf{A3} but \textit{that keeps the regression function $\eta$ in our smoothness class $\mathcal{F}_{L,\alpha}$}. 
This is a major difference with former counter examples built in \cite{DGL} where the non uniform consistency is obtained with a family  of \textit{non-smooth} regression functions $\eta$. In our study, we also obtained a family of \textit{smooth} regression functions for which such phenomena occur.
Even in this case, it is still possible to keep the excess risk strictly positive for any classifier $\Phi_n$ (and no longer for only nearest neighbor rules).

\vspace{0.5em}
\qquad Finally, it should be noted that our inconsistency results always occur when building a network of regression functions $\eta$ that oscillate around the value $1/2$ at the neighborhood of the set  $\left\{\mu=0 \right\}$. In a sense, Theorem \ref{theo:NON1} contributes to the understanding of one of the opens question put forth in \cite{Cannings} on the behavior of the nearest neighbor rule when $\eta$ is oscillating about $1/2$ in the tail.

\subsection{Minimax rates of convergence}\label{sec:minimax8}

In the meantime, when both \textbf{A2}, \textbf{A3} and \textbf{A4} hold, we are able to precisely describe the corresponding minimax rate of convergence.

\subsubsection{Minimax lower bound}

\begin{theo}\label{theo:NON2} Assume that Assumptions \textbf{A1-A4} hold. Then
$$
\inf_{\Phi_n} \sup_{\mathbb{P}_{(X,Y)} \in \mathcal{F}_{L,\alpha} \cap \blabla \cap \PT} \left[ \mathcal{R}(\Phi_n)- \mathcal{R}(\Phi^{*}) \right] \gtrsim n^{-\frac{1+\alpha}{2+\alpha+d}}.
$$
\end{theo}
 For the sake of convenience, we  briefly outline the proof of Theorem \ref{theo:opt} borrowed from \cite{AudTsy}  in Section \ref{sec:lowerbound}. It is then adapted to our new set of assumptions. 

\vspace{0.5em}
\qquad Theorem \ref{theo:Gene1} below provides some lower bounds for different tails of distributions (through the function $\psi$). 
It should be noted that we recover the known rate of compactly supported densities with the so-called \textit{Mild Density Assumption} of \cite{AudTsy} in the particular case $\psi = Id$. 
This implies that in the non-compact case, the rate cannot be improved compared to the compact setting, even with an Additional Tail assumption.

\subsubsection{An upper bound for the nearest neighbor rule}\label{sec:sliced}

When the density is no longer bounded away from $0$, the integer $k_n$ will be chosen in order to counterbalance  the vanishing probability of the small balls in the tail of the distributions. For example, when $\psi=Id$, we show that a suitable choice of the integer $k_n$ is:
$$
k_n :=\lfloor  n^{\frac{2}{3+\alpha +d}} \rfloor ,
$$
which appears to be quite different from the one in the previous section. 

\begin{theo}\label{theo:Main2} Assume that \textbf{A1}-\textbf{A3} hold and if the Tail Assumption \textbf{A4} is driven by $\psi=Id$, the choice $k_n :=\lfloor  n^{\frac{2}{3+\alpha +d}} \rfloor $ yields:
$$
\sup_{\mathbb{P}_{(X,Y)} \in \mathcal{F}_{L,\alpha} \cap  \PT \cap \blabla} \left[ \mathcal{R}(\Phi_{n,k_n})- \mathcal{R}(\Phi^{*}) \right] \lesssim  n^{-\frac{(1+\alpha)}{(3+\alpha+d)}}.
$$

\end{theo}

The proof of Theorem \ref{theo:Main2} is provided in Section \ref{s:preuve_noncompact}. The above results indicate that the price to pay for the classification from entries in compact sets to arbitrary large sets of $\mathbb{R}^d$ is translated by the degradation from $n^{-(1+\alpha)/(2+d)}$ to at least $n^{-(1+\alpha)/(2+\alpha+d)}$ (see, \textit{e.g.}, Theorem \ref{theo:NON2} when $\psi(\epsilon)\sim \epsilon$). Our upper bound for the nearest neighbor rule does not exactly match this lower bound  since we obtain $n^{-(1+\alpha)/(3+\alpha+d)}$ in a similar situation . At this step, obtaining the appropriate minimax rate requires slight changes inside the construction of the nearest neighbor rule. This is the purpose of the next paragraph.

\subsubsection{Minimax upper bound for an optimal nearest neighbor rule}\label{sec:minimax113}

The upper bound proposed in the theorem can be improved if we change the way in which the regularization parameter $k_n$ is constructed. We use a nearest neighbor algorithm with a number of neighbors that depends on the position of the observation $x$ according to the value of the density $\mu(x)$. More formally, we define for all $j\in \mathbb{N}$
$$
\Omega_{n,0} := \left\lbrace x \in \mathbb{R}^d :  \   \mu(x) \geq n^{\frac{-\alpha}{2+\alpha+d}}  \right\rbrace, $$
and
$$ \Omega_{n,j} = \left\lbrace x \in \mathbb{R}^d :  \  \frac{n^{\frac{-\alpha}{2+\alpha+d}}}{ 2^{j}} \leq \mu(x) < \frac{n^{\frac{-\alpha}{2+\alpha+d}}}{ 2^{j+1}} \right\rbrace.
$$
Setting $k_{n,0} = \lfloor n^{\frac{2}{2+\alpha+d}} \log(n) \rfloor $, we then use for all $j\in \mathbb{N}$
\begin{equation}
k_{n}(x)=\lfloor k_{n,0} 2^{-2j/(2+d)}\rfloor\vee 1  \quad \text{when} \quad x \in \Omega_{n,j}.
\label{eq:optknn}
\end{equation}

According to (\ref{eq:optknn}), the number of neighbors involved in the decision process depends on the spatial position of the input $X$. In some sense, this position is linked to the tail. The statistical performances of the corresponding nearest neighbor classifier is displayed below. Such a construction of this sequence of ``slices" may be interpreted as a spatial adaptive bandwidth selection. This bandwidth is smaller at points $x \in \RR^d$ such that $\mu(x)$ is small. In a sense, this idea is close to the one introduced in \cite{Lepski} that provides a similar slicing procedure to obtain an adaptive minimax density estimation on $\RR^d$.


\begin{theo}\label{theo:Main_opt}  Assume that \textbf{A1}-\textbf{A3} hold and that the Tail Assumption \textbf{A4} is driven by $\psi=Id$. Then, if $\Phi^*_{n,k_n}$ is the classifier associated with (\ref{eq:optknn}), we have:
$$
\sup_{\mathbb{P}_{(X,Y)} \in \mathcal{F}_{L,\alpha} \cap  \PT \cap \blabla} \left[ \mathcal{R}(\Phi^{*}_{n,k_n})- \mathcal{R}(\Phi^{*}) \right] \lesssim  n^{-\frac{(1+\alpha)}{(2+\alpha+d)}}  (\log n)^{\frac{1}{2}+\frac{1}{d}}.
$$
\end{theo}

We stress that the upper bound obtained in Theorem \ref{theo:Main_opt} nearly matches the lower bound proposed in Theorem \ref{theo:NON2}, up to a log-term. This log-term can be removed by the use of additional technicalities that are omitted in our proof. Hence, Theorems \ref{theo:Main_opt} and \ref{theo:NON2} make it possible to identify the exact minimax rate of classification when the Tail Assumption is driven by $\psi=Id$, that is:
$$
\inf_{\Phi} \sup_{\mathbb{P}_{(X,Y)} \in \mathcal{F}_{L,\alpha} \cap  \PT \cap \blabla}  \left[ \mathcal{R}(\Phi^{*}_{n,k_n})- \mathcal{R}(\Phi^{*}) \right]  \sim n^{-\frac{1+\alpha}{2+\alpha+d}}.
$$

\subsubsection{Generalizations}  

We propose several extensions of our previous results (lower and upper bounds) for more general tails of distribution. We also propose to enlighten the \textit{Minimal Mass Assumption} $\blabla$.

\paragraph{Effect of the tail: from  $\PT$ to $\PTf$}

\begin{theo}\label{theo:Gene1} 
Assume that Assumptions \textbf{A1-A4} hold. For any tail $\mathcal{T}$ parameterized by a function $\psi$, we obtain the following results:
\begin{itemize}
\item[i)] \underline{Lower bound:} the minimax classification rate satisfies:
$$
\inf_{\Phi_n} \sup_{\mathbb{P}_{(X,Y)} \in \mathcal{F}_{L,\alpha} \cap \PTf \cap \blabla } \left[ \mathcal{R}(\Phi_n)- \mathcal{R}(\Phi^{*}) \right] \gtrsim \epsilon_{n,\alpha,d}^{1+\alpha},
$$
where $\epsilon_{n,\alpha,d}$ satisfies the balance
\begin{equation}\label{eq:ratelow}
n^{-1} = \{\epsilon_{n,\alpha,d}\}^{2+d} \times \psi^{-1}\left(\{\epsilon_{n,\alpha,d}\}^\alpha\right).
\end{equation}
\item[ii)] \underline{Upper bound:} the nearest neighbor rule satisfies
$$
\sup_{\mathbb{P}_{(X,Y)} \in \mathcal{F}_{L,\alpha} \cap  \PTf \cap \blabla} \left[ \mathcal{R}(\Phi_{n,k_n})- \mathcal{R}(\Phi^{*}) \right] \leq C \nu_{n,\alpha,d}^{1+\alpha} 
$$
with $k_n = \nu_{n,\alpha,d}^{-2}$ where $\nu_{n,\alpha,d}$ fulfills the balance:
\begin{equation}\label{eq:rates}
n^{-1} =  \psi^{-1}(\{\nu_{n,\alpha,d}\}^{1+\alpha}) \{\nu_{n,\alpha,d}\}^{2+d}.
\end{equation}
\end{itemize}
\end{theo}

It would also be possible to propose some generalizations using the sliced nearest neighbor rule presented in Sections \ref{sec:sliced} and \ref{sec:minimax113} for  tails driven by a general function $\psi$, even if we do not include this additional result for the purpose of clarity.

\paragraph{Meeting the Minimal Mass Assumption $\bla$}
We now obtain similar rates when using the weaker assumption $\bla$ instead of $\blabla$: the lower bounds of $\mu(B(x,\delta))$ are only useful for some  points $x$ such that $\mu(x)$ is large enough. We can state the next corollary.
\begin{cor} Assume that \textbf{A1},\textbf{A2},\textbf{A4} hold and  $P_{(X,Y)} \in \bla$, then
$$
\sup_{\mathbb{P}_{(X,Y)} \in \mathcal{F}_{L,\alpha} \cap  \PTf \cap \bla} \left[ \mathcal{R}(\Phi_{n,k_n})- \mathcal{R}(\Phi^{*}) \right] \lesssim \nu_{n,\alpha,d}^{1+\alpha} ,
$$
with $k_n = \nu_{n,\alpha,d}^{-2}$ where $\nu_{n,\alpha,d}$ satisfies the balance
$$
n^{-1} =  \psi^{-1}(\{\nu_{n,\alpha,d}\}^{1+\alpha}) \{\nu_{n,\alpha,d}\}^{2+d}.
$$
\end{cor}

 The condition $\blabla$ cannot be easily described through an analytical condition because of its uniform nature over $\Omega$. In contrast, $\bla$ is more tractable in view of the criterion given by the next result (Proposition \ref{prop:sufficient}). Using a log-density model, we write the density $\mu$ as
$$
\mu(x) =e^{-\varphi(x)}, \forall x \in \RR^d.
$$

\begin{prop}\label{prop:sufficient}
Let $\varphi \in \mathcal{C}^{1}(\Omega)$ and assume that a real number $a>0$ exists  such that:
$$
\lim_{ x : \mu(x) \longrightarrow 0}  \frac{\| \nabla \varphi(x) \| }{\varphi(x)^a} = 0,
$$
then a suitable $\kappa$ can be found such that $\mu=e^{-\varphi} \in \bla$.
 \end{prop}

\begin{proof}
For any $\delta >0$, we compute a lower bound of
$$
\mathbb{P}_X\left(B(x,\delta)\right) = \int_{B(x,\delta)} e^{- \varphi(z)} dz.
$$
The Jensen Inequality applied to the normalized Lebesgue measure over $B(x,t)$, which is denoted $\bar{d}z$,   yields
\begin{equation}\label{eq:tempo}
 \int_{B(x,\delta)} e^{- \varphi(z)} dz \geq \frac{\pi^{d/2} \delta^{d}}{\Gamma(d/2+1)} \exp \left( - \varphi(x) + \int_{B(x,\delta)} [\varphi(z) - \varphi(x)] \bar{d}z  \right).
\end{equation}
A first order Taylor expansion leads to
$$
\int_{B(x,\delta)} [\varphi(z) - \varphi(x)] \bar{d}z  \leq \sup_{z \in B(x,\delta)} \|\nabla \varphi(z)   \| \int_{{B(x,\delta)}} \|z-x\| \bar{d}z  \leq    \delta \sup_{z \in B(x,\delta)} \|\nabla \varphi(z)   \|.
$$
Now, our assumption on $\varphi$ implies that a large enough $C_a$ exists such that:
$$
\|\nabla \varphi(z)\| \leq C_a (1+\varphi(z)^a).
$$
Thus, the lower bound \eqref{eq:tempo} becomes:
$$
\int_{B(x,\delta)} e^{- \varphi(z)} dz \geq \frac{\pi^{d/2} \delta^{d}}{\Gamma(d/2+1)} e^{ - \varphi(x) }  e^{-C_a \delta (1+\sup_{z \in B(x,\delta)} \varphi^a(z))}.
$$
It is now sufficient to consider points $x$ such that $\varphi \leq \delta^{-1/a}$ (equivalent to $\mu \geq e^{-\delta^{-1/a}}$) to obtain a meaningful lower bound Hence, 
$\bla$ is satisfied choosing 
$$\rho=1/a  \qquad \text{and} \qquad \kappa = \frac{\pi^{d/2}}{2 \Gamma(d/2+1)} e^{-C_a}.$$
\end{proof}

\subsection{Practical settings on typical examples \label{sec:examples}}

The aim of this section is to illustrate the results obtained above. We first describe a location model for which we can derive explicit upper and lower bounds in several different cases. We then propose a small numerical study in order to enhance the discussion regarding the importance of the Tail Assumption and we conclude by drawing a comparison between the standard nearest neighbor  and sliced nearest neighbor rules.

\paragraph{Explicit rates for specific location models}

We investigate here the influence of the function $\psi$ in $\PTf$ as well as the one of the margin parameter on the convergence rates through several specific location models. These models are defined as follows: given any positive random variable $Z$
(whose cumulative distribution function is denoted as $F$) and two real location values $a$ and $b$, the random variable $X$ is given by:
\begin{equation}\label{eq:model_location}
X=\epsilon Z + Y b+(1-Y)a,
\end{equation}
where $\epsilon$ is a Rademacher random variable (whose values is $\pm 1$) independent of $Z$, and $Y$ is the label of the observation, sampled independently of $\epsilon$ and $Z$ with a Bernoulli law $\mathcal{B}(1/2)$. Using a translation invariance argument, it is enough in the next study to consider $a=0$ and $b>0$. Table \ref{table:gloss} illustrates the rate reached by the nearest neighbor procedure in each situation.

\begin{table}[h!]
\begin{center}
\begin{tabular}{|l||c|c|c|c|}
\hline 
Law of $Z$ &Tail $\psi$ & Margin & $k_n \sim n^{\beta}$ &  Upper bound\\
\hline 
Gauss & $\psi(\epsilon) \propto \epsilon \log(1/\epsilon)^{r/2-1}$ & $\alpha = 1$ & $ \beta = 2/(4+d)$ & $n^{-2/(4+d)} \log(n)^{\beta(r)}$\\ \vspace{0.2em}
Laplace & $\psi(\epsilon) \propto\epsilon$& $\alpha=1$ & $ \beta = 2/(4+d)$ &$n^{-2/(4+d)}$\\\vspace{0.2em}
Gamma & $\psi(\epsilon) \propto \epsilon \log(1/\epsilon)^{k-1}$& $\alpha=1$ & $ \beta = 2/(4+d)$ & $n^{-2/(4+d)} \log(n)^{\beta(k)}$\\\vspace{0.2em}
Cauchy & $\psi(\epsilon) \propto\sqrt{\epsilon}$&$\alpha=1$& $ \beta = 1/(3+d)$&$n^{-2/(3+d)}$\\\vspace{0.2em}
Power-Pareto & $ \psi(\epsilon) \propto \epsilon^{p/(p+1)}$ & $\alpha= 1 \wedge p$
&$ \beta = \frac{2(p+1)}{p(3+\alpha+d)+2+d}$ & $n^{\frac{-4(p+1)}{p(3+\alpha+d)+2+d}}$\\
\hline
\end{tabular}
\caption{\label{table:gloss}Convergence rates for location models with several tail sizes.}
\end{center}
\end{table}

\paragraph{A numerical study for 'power laws'}  In order to illustrate Equations \eqref{eq:ratelow} and \eqref{eq:rates}, we consider some specific cases of  ``power laws" such that:
$$
\mathbb{P}_X(\mu(X)<\epsilon) = \psi(\epsilon) \sim \epsilon^g \quad \text{when} \quad \epsilon \longrightarrow 0^{+},
$$
for some $g>0$. In this case, the upper bound on the Nearest Neighbor classifier is given by
$$
 \mathcal{R}(\Phi_n) - \mathcal{R}(\Phi^{*}) \lesssim n^{-\frac{(1+\alpha)}{1+\alpha + \frac{2+d}{g}  }}
$$
although the lower bound derived from \eqref{eq:ratelow} is:
$$
\inf_{\Phi_n} \sup_{\mathbb{P}_{(X,Y)} \in \mathcal{F}_{L,\alpha} \cap   \PTf \cap \blabla} \left[ \mathcal{R}(\Phi_n)- \mathcal{R}(\Phi^{*}) \right] \gtrsim n^{-\frac{(1+\alpha)}{\alpha + \frac{2+d}{g}}}
$$
We immediately observe that the classification rates are seriously damaged when $g$ is small. In contrast, for very thin tails, the rate can be arbitrarily close to $n^{-1}$. For this purpose, we illustrate this phenomenon with a family of distributions $\mathcal{P}_{g}$, where the parameter $g>0$ influences the tail size. 
We define the cumulative distribution function of the positive random variable $Z$:
$$
\forall t \geq 0 \qquad F_g(t) = 1-\frac{1}{(t+1)^{g}}.
$$
Then, for two real values $(a,b)$, we sample $n$ observations $(X_i,Y_i)$ according to the previous model and the Bayes classifier is given by:
$$\Phi^{*}(X) = \mathbf{1}_{\lbrace X>(a+b)/2\rbrace}.$$

In this example, the margin $\alpha$ is equal to $1$ and $\eta$ is $L$-Lipschitz. We then consider $k_n = \lfloor n^{2/5}\rfloor +1$
to assess the statistical performance of the Nearest Neighbor classifier.
Figure \ref{fig:rates} represents the excess risk obtained by the Nearest Neighbor classifier and the
successive degradation of the convergence rate when $g$ decreases to $0$ (on the left, the empirical performance of the Nearest Neighbor rule with the underlying distributions  and on the right for the  upper bound theoretically derived from Theorem \eqref{theo:Main2}). These numerical experiments are consistent with the theoretical result obtained in Theorem \ref{theo:Gene1} .

\begin{figure}
\begin{center}
\includegraphics[scale=0.35]{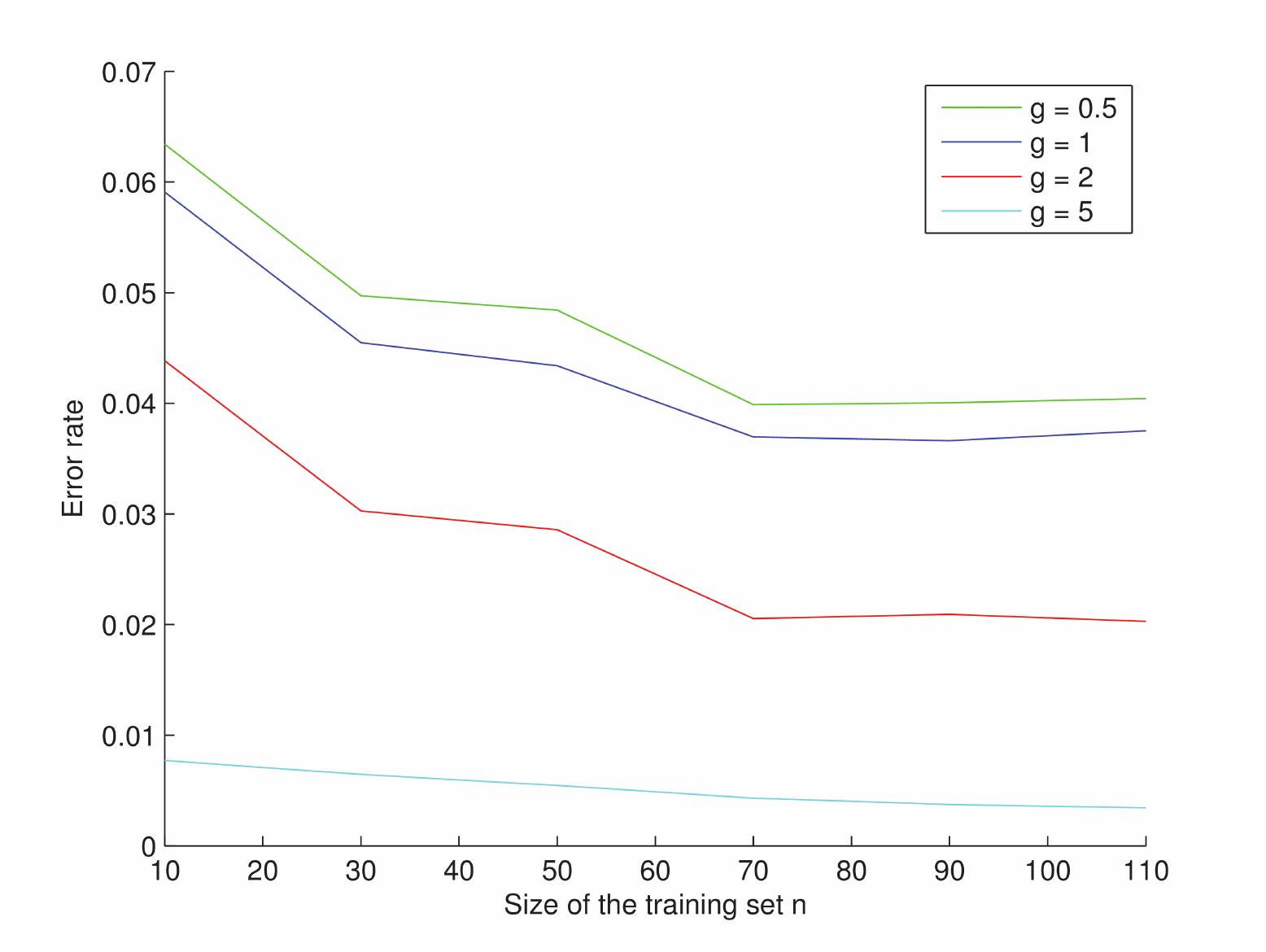}
\includegraphics[scale=0.35]{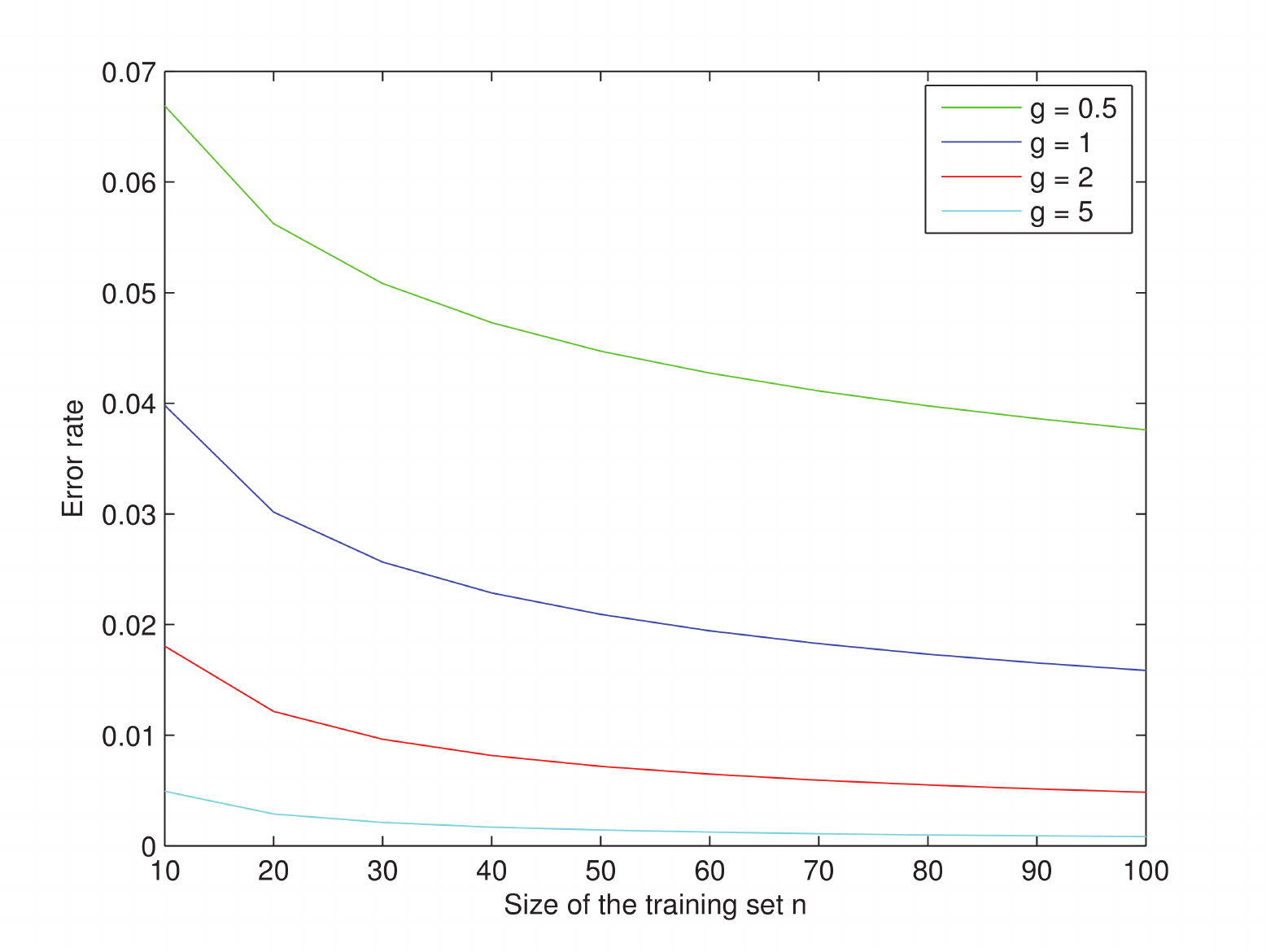}
\end{center}
\caption{Example of observed empirical rates  and upper bound theoretical rates given by \eqref{eq:rates} for several power law distributions of parameter $g$.\label{fig:rates}}
\end{figure}

\paragraph{Comparison between the standard nearest neighbor and its sliced counterpart}

We provide here a short numerical study that aims to compare the results reached by the standard nearest neighbor rule described in Theorem \ref{theo:Main2} and the ones obtained by its sliced counterpart described in Section \ref{sec:minimax113} and in Theorem \ref{theo:Main_opt}.
To measure such an improvement, we have chosen to once again use some non-compactly supported distributions and several different location models.

On the one hand, as pointed out in Theorem \ref{theo:Main2}, the standard nearest neighbor will be tuned with a number of neighbor $k_n :=\lfloor  n^{\frac{2}{3+\alpha +d}} \rfloor +1$.

On the other hand, the sliced nearest neighbor rule described in Theorem \ref{theo:Main_opt} requires a preliminary estimation of the law of observation $\mathbb{P}_X$. To do this, we used the recent kernel density estimation package\footnote{kde.m is available on the author's Website of \cite{kernel_diffusion}.} provided by \cite{kernel_diffusion}, which is an adaptive estimator based on linear diffusion processes. Given any training set $(X_i,Y_i)_{1 \leq i \leq n}$, we first built the preliminary estimator $\hat{\mu}_n$ of the unknown density $\mu$. This estimator is interesting because of its adaptive smoothing properties and because it includes a very fast automatic bandwidth selection algorithm.

The sliced nearest neighbor rule then uses a number of neighbors that depends on the design point $X$. If the density estimate is large enough, that is, if $\hat{\mu}_n(X) \geq n^{-\frac{\alpha}{2+\alpha+d}}$:
$$
k_n(X) :=\lfloor  n^{\frac{2}{2+\alpha +d}} \rfloor +1.
$$
Otherwise, when $ 2^{-(j+1)} \leq \hat{\mu}_n(X) n^{\frac{\alpha}{2+\alpha+d}} \leq 2^{-(j)}$, the number $k_n(X)$ is:
$$
k_n(X) := \lfloor  n^{\frac{2}{2+\alpha +d}} 2^{- \frac{2 j}{2+d}} \rfloor +1.
$$

To draw some reliable comparisons, we also used some various laws for the random variable $Z$ involved in the definition of the location model \eqref{eq:model_location} (Normal distributions, Cauchy distributions, and Power laws) whose parameters are described in Table \ref{table:gloss}. The two location parameters are still denoted $a$ and $b$ and fixed such that $a=-b$.

In each situation, we used a Monte-Carlo strategy with 1000 replications to compute the mean excess risk of each nearest neighbor rule. We used a training set of cardinal $n$, as well as a test set of size $200$. Results are given in Table \ref{table:slice_vs_std}.

\begin{table}[h!]
\begin{center}
\begin{tabular}{|l|c|c|c|c|c|c|c|c|c|}
\hline
Law of $Z$ & \multicolumn{3}{c|}{$n=100$} & \multicolumn{3}{c|}{$n=500$} & \multicolumn{3}{c|}{$n=1000$} \\
\hline
Gauss, $a=1,\sigma=2$ &$19.2_{.6} $ &$18.1_{.6}$&6$\%$ &$16.4_{.5}$&$13.9_{.5}$&15$\%$&$15.4_{.5}$&$12_{.5}$&22$\%$ \\\hline \vspace{0.2em}
Cauchy, $a=\frac{1}{2},\gamma=\frac{1}{2}$ & $2.6_{.2}$&$1.9_{.2}$&26$\%$&$1.4_{.1}$&$1.2_{.1}$&14$\%$&$0.9_{.05}$&$0.8_{.05}$&6$\%$ \\\hline \vspace{0.2em}
Cauchy, $a=\frac{1}{2},\gamma=1$ & $4.4_{.3}$&$3.6_{.2}$&18$\%$&$3.1_{.3}$&$2.2_{.2}$&28$\%$&$2.3_{.2}$&$1.4_{.2}$&37$\%$ \\\hline \vspace{0.2em}
Power, $a=\frac{1}{2},\gamma=1$ & $3.8_{.3}$&$3_{.3}$&20$\%$&$2.7_{.2}$&$2.1_{.2}$&22$\%$&$1.9_{.2}$&$1.5_{.1}$&19$\%$ \\\hline \vspace{0.2em}
Power, $a=\frac{1}{2},\gamma=2$ & $2_{.2}$&$1.7_{.2}$&13$\%$&$1.2_{.2}$&$1.0_{.1}$&15$\%$&$0.7_{.1}$&$0.6_{.1}$&14$\%$ \\\hline
\end{tabular}
\end{center}
\caption{ Mean excess risk multiplied by $100$ (left: standard nearest neighbor; middle: sliced nearest neighbor; right: percentage of improvement). Standard errors are given in small script. \label{table:slice_vs_std}}
\end{table}
We may observe in Table \ref{table:slice_vs_std} that the sliced version of the nearest neighbor always outperforms the standard one. Such a numerical result is consistent with the theoretical ones of Theorem \ref{theo:Main2} and  \ref{theo:Main_opt}.
Note also that the relative improvement of the sliced nearest neighbor rule seems to increase when the number of observations $n$ growth, meaning that each excess risk of the two procedures varies with a different power of $n$.

Finally, it should be mentioned that we have not tried to modify the dimension of the observations $X$. Indeed, the difference of the upper bounds given by Theorems \ref{theo:Main2} and \ref{theo:Main_opt} becomes more and more negligible when the dimension is increasing. This should also be  the case in the empirical study that will  be in the subject of a future work. Likewise, the statistical study of the empirical sliced nearest neighbor rule should also be addressed in a future study, since a balance between the estimation $\hat{\mu}_n$ of the density $\mu$ and the excess risk of classification with the sliced rule may exist. We have left this problem open for a future study.
%
%
%


\begin{supplement}[id=suppA]
  \sname{Supplement A}
  \stitle{Main proofs for this paper : Classification with the nearest neighbor rule in general finite dimensional spaces: necessary and sufficient conditions.}
  \slink[doi]{COMPLETED BY THE TYPESETTER}
  \sdatatype{.pdf}
  \sdescription{See in the temporary Appendix section after references.}
\end{supplement}

\bibliographystyle{alpha}
\bibliography{vraibib}

\newpage 
\appendix
\section{Proofs}
\label{s:proofs}

Recall that $\EE$ (resp. $\EE_X$, $\EE_{\otimes^n}$) denote the expectation with respect to the measure $\PP$ (resp. $\PP_X$, $\PP_{\otimes^n}$).

\subsection{Proofs of the lower bounds} \label{sec:lowerbound}

The proofs of the lower bounds presented in both Theorem \ref{theo:NON1} and Theorem \ref{theo:NON2} are inspired from the construction proposed in  \cite{AudTsy}.
It is based on Assouad's cube method (see \cite{Assouad1}, and  \cite{Assouad2}). This approach reduces the problem of obtaining a lower bound on the minimax risk to the problem of testing several couples of hypotheses. We refer to \cite{Tsybakov_book} for a comprehensive introduction to this useful method for deriving lower bounds on minimax risk.

\subsubsection{Baseline structure of the network}\label{sec:baseline}
We present here the common structure of the network of laws on $(X,Y)$, that is, the definition of the underlying measure $\mathbb{P}_{(X,Y)}$ on $\RR^d \times \{0,1\}$ (through the density $\mu$ and the regression function $\eta$).

\paragraph{Definition of $\eta$}
Let $(q,m)\in(\NN^{*})^{2}$ and $\left(x_1,\ldots,x_{m}\right)\in\RR^{d}$. We denote by $B_{i}$ the Euclidean ball of center $x_{i}$ and of radius  $2/q$, such that for any $i$ and $j$ we have $B_i\bigcap B_{j}=\emptyset$ (we choose $|x_{i}-x_{j}|\geq 5/q$).  Now consider a $C^{\infty}$ function $\varphi$ such that $\|\varphi\|_{\infty} = 1$, $\varphi$ is compactly supported in $[0, 2]$ such that $\varphi(x)=1$ when $|x| \leq 1$, and $\varphi(x)= 0$ for any $x>3/2$. Now let $\Phi_{j}(x)=c_{\varphi}q^{-1}\varphi(q|x-x_{j}|)$ so that  $\Phi_{j}(x)$ is also $C^{\infty}$ and supported in $B_j :=B(x_j,2/q)$. 
Denote by $A_{0}=\bigcup_{j=1}^{m}B_{j}$ and let $A_1=[0,1]^d\bigcap A_0^c$ and $A=A_0\bigcup A_1$ be  the support of the density $\mu$.


\paragraph{Definition of the Assouad Hypercube of regression functions}
We define
$
\Sigma_{m}=\left\{-1,1\right\}^{m},
$
and for any $\sigma\in \Sigma_{m}$: 
$$\forall 1 \leq j \leq m,\ \forall x\in B_{j}:\ \eta_{\sigma}(x)=\frac{1+\sigma_{j}\Phi_{j}(x)}{2}, \quad \text{and} \quad \eta_{\sigma}(x)=\frac{1}{2} \, \text{if} \, x\in A_{1}.$$
Figure \ref{fig:eta_sigma} shows  the regression function $\eta_{\sigma}$ for two opposite values of $\sigma_j$ and for a particular ball $B_j$.
\begin{figure}
\begin{center}
\includegraphics[scale=0.25]{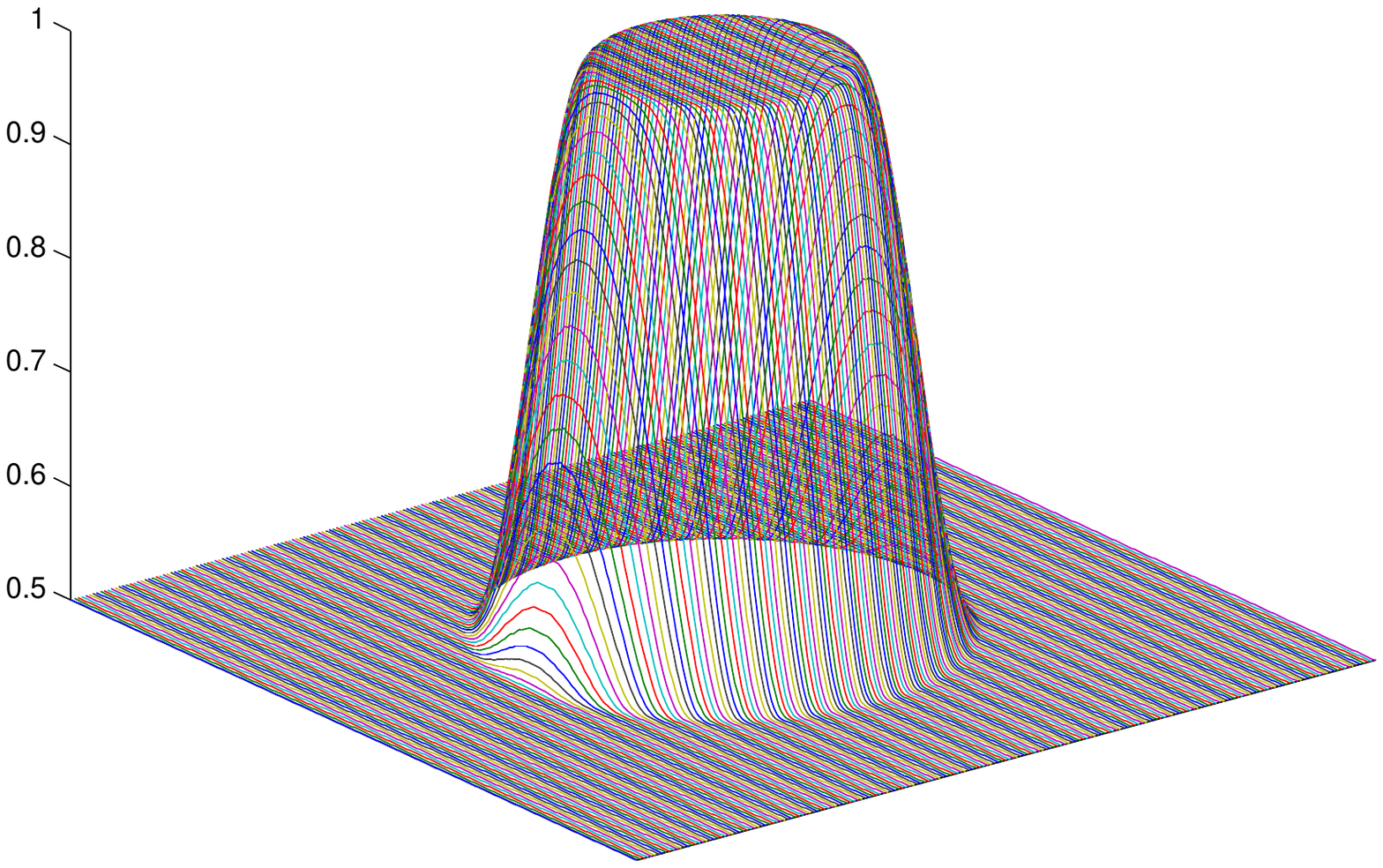}
\includegraphics[scale=0.25]{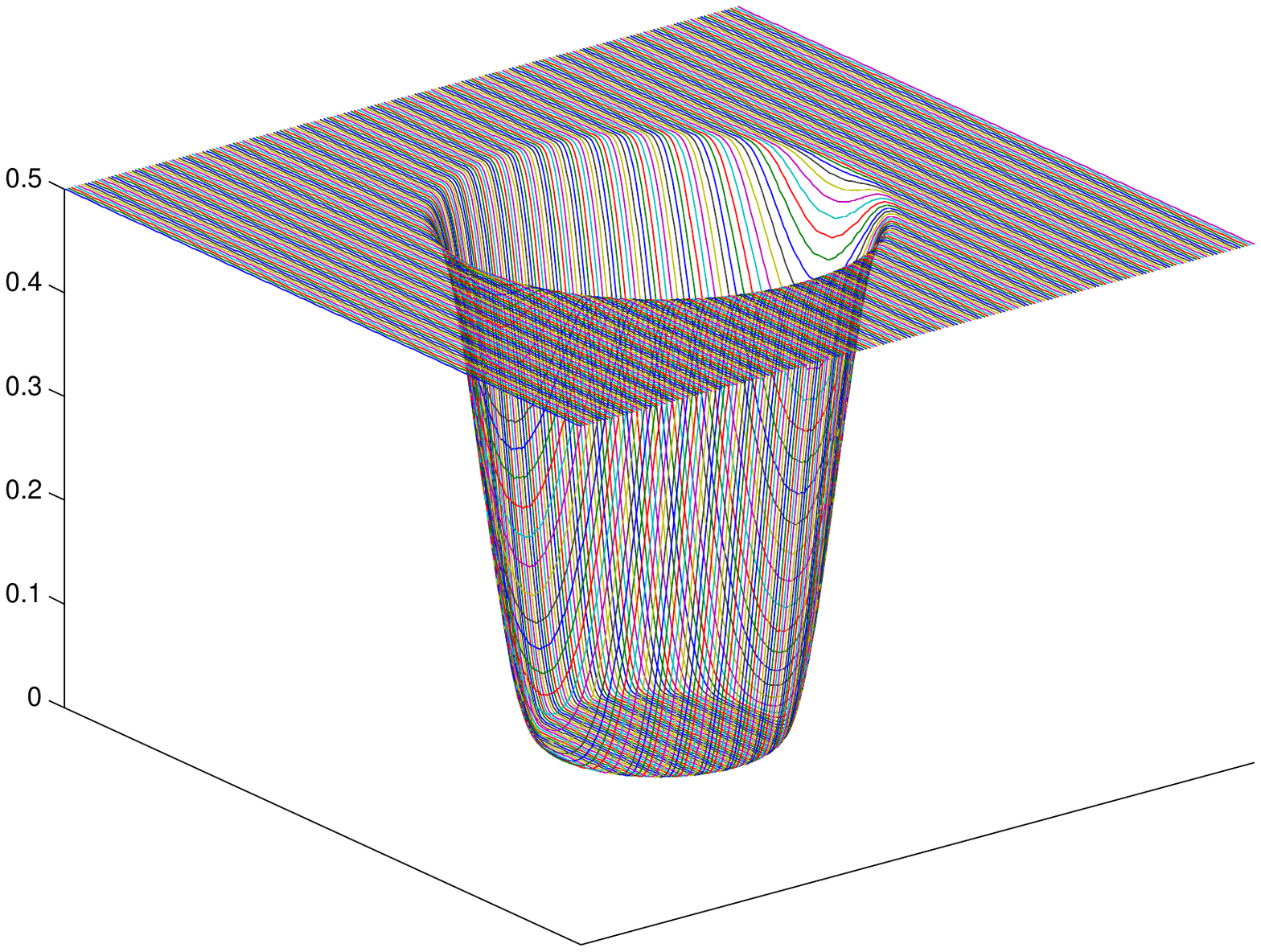}
\end{center}
\caption{Example of function $\eta_{\sigma}$ on a particular ball $B_j$ of size $1/q$. The value of $\eta_{\sigma}$ oscillates either  between $1/2$ and $1$ if $\sigma_j=1$ or $0$ and $1/2$ if $\sigma_j=-1$.\label{fig:eta_sigma}}
\end{figure}

\paragraph{The density $\mu$}
We use in the sequel a measure $\mu$ in the sequel that does not depend on $\sigma$.  Indeed, we even consider only some constant densities on each $B_{j}$. In particular, the measure $\mu$ of each ball $B_j$ is $\omega$ (that will be chosen later) and the density $\mu$ is then given by
$$\mu(x)=\frac{\omega}{\lambda(B_{j})} = \frac{ \omega q^d}{\gamma_d2^d},\ \forall x:in B_j$$
where $\gamma_d$ is the Lebesgue measure of the unit Euclidean ball of $\RR^d$. We now define $\mu$ on $A_1$ as
$$\mu(x)=\frac{1-m\omega}{\lambda\left(A_{1}\right)}.$$  A schematic representation of this measure
can be seen on the left of Figure \ref{fig:mesure}.

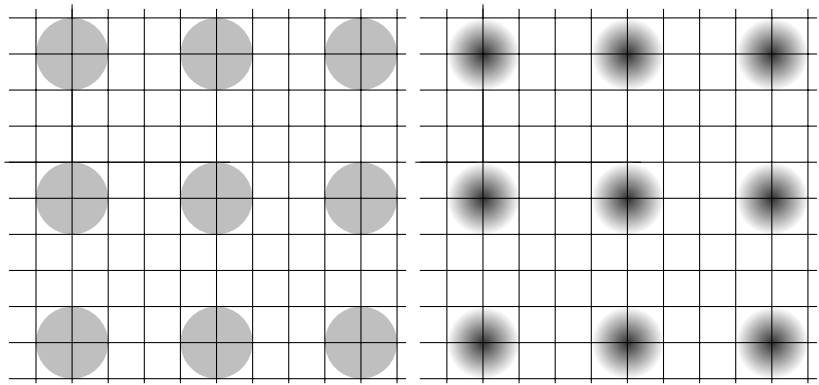
\begin{figure}[h!]
\begin{center}
\begin{tikzpicture}[scale=0.6]
\draw (-1.5,0) -- (3.5,0);
\draw (0,-1.5) -- (0,3.5);
\shade[inner color=gray!50, outer color=gray!50] (-0,2.4) circle (0.8cm);
\shade[inner color=gray!50, outer color=gray!50] (3.2,2.4) circle (0.8cm);
\shade[inner color=gray!50, outer color=gray!50] (6.4,2.4) circle (0.8cm);
\shade[inner color=gray!50, outer color=gray!50] (-0,-0.8) circle (0.8cm);
\shade[inner color=gray!50, outer color=gray!50] (3.2,-0.8) circle (0.8cm);
\shade[inner color=gray!50, outer color=gray!50] (6.4,-0.8) circle (0.8cm);
\shade[inner color=gray!50, outer color=gray!50] (-0,-4) circle (0.8cm);
\shade[inner color=gray!50, outer color=gray!50] (3.2,-4) circle (0.8cm);
\shade[inner color=gray!50, outer color=gray!50] (6.4,-4) circle (0.8cm);
\draw[step=.8cm] (-1.4,-4.9) grid (7.4,3.4);
\end{tikzpicture}
\begin{tikzpicture}[scale=0.6]
\draw (-1.5,0) -- (3.5,0);
\draw (0,-1.5) -- (0,3.5);
\shade[inner color=black!80, outer color=white] (-0,2.4) circle (0.8cm);
\shade[inner color=black!80, outer color=white] (3.2,2.4) circle (0.8cm);
\shade[inner color=black!80, outer color=white] (6.4,2.4) circle (0.8cm);
\shade[inner color=black!80, outer color=white] (-0,-0.8) circle (0.8cm);
\shade[inner color=black!80, outer color=white] (3.2,-0.8) circle (0.8cm);
\shade[inner color=black!80, outer color=white] (6.4,-0.8) circle (0.8cm);
\shade[inner color=black!80, outer color=white] (-0,-4) circle (0.8cm);
\shade[inner color=black!80, outer color=white] (3.2,-4) circle (0.8cm);
\shade[inner color=black!80, outer color=white] (6.4,-4) circle (0.8cm);
\draw[step=.8cm] (-1.4,-4.9) grid (7.4,3.4);
\end{tikzpicture}
\end{center}
\caption{Simplified representation of the measure $\PP_X$, the gray level is proportional to the value of the density $\mu$. Left: measure used in Section \ref{sec:baseline} or \ref{sec:proofNON2} (compactly supported measure or not) and in Section \ref{sec:mmasanstail} ($\blabla$ is fulfilled and not the Tail Assumption). Right: measure used in Section \ref{sec:sansmma} when the tail is fulfilled and not $\blabla$. \label{fig:mesure}}
\end{figure}

\paragraph{Margin condition}
For the sake of convenience, for any $\sigma \in \Sigma_m$, we denote $\mathbb{P}_{\sigma} := \mathbb{P}_{(X,Y),\sigma}$ the law of the couple $(X,Y)$.
Following the arguments of  \cite{AudTsy}, consider any $\sigma \in \Sigma_m$:
\begin{eqnarray*}
\mathbb{P}_{\sigma}\left(0 <|\eta_{\sigma}(X) - \frac{1}{2}| \leq t \right) &=&m  \mathbb{P}_{\sigma} \left( 0 <c_\varphi\varphi(q[|X-x_1|]) \leq 2 t q \right) \\
& = & m \int_{B(x_1,2/q)} \mathbf{1}_{\{0 \leq c_\varphi\varphi(q[|x-x_1|]) \leq 2 t q\}} \mu(x) dx
\end{eqnarray*}
Since $\varphi$ is equal to $1$ on $[0,1]$, we then obtain that:
\begin{eqnarray*}
\mathbb{P}_{\sigma}\left(0 < |\eta_{\sigma}(X) - \frac{1}{2}| \leq t \right) & \leq& m \int_{B(x_1,2/q)}  \mathbf{1}_{\{c_\varphi \leq 2 t q\}} \mu(x) dx \\
& = & \mathbf{1}_{\{c_\varphi \leq 2 t q\}} m \omega \lesssim  t^{\alpha}.
\end{eqnarray*}
as soon as:
$$
m\omega=O(q^{-\alpha}).
$$

\paragraph{Smoothness of $\eta_{\sigma}$}
We briefly check that the regression functions are Lipschitz, uniformly with respect to any choice of $q$.
First, it should be observed that:
$$
\forall (x,\tilde{x}) \in B_j \qquad |\eta_{\sigma}(x) - \eta_{\sigma}(\tilde{x})| = \frac{|\Phi_j(x)-\Phi_j(\tilde{x})|}{2} \leq \frac{ c_{\varphi}  \|\varphi'\|_{\infty}}{2}\|x-\tilde{x}\|
$$
On the contrary, when $(x,\tilde{x}) \in A_1$, $\eta_{\sigma}(x)=\eta_{\sigma}(\tilde{x})=1/2$. It now remains to study the situation when $x\in A_1$ and $\tilde{x} \in B_j$ for one $j$. When $\tilde{x}$ is in the exterior ring of size $3/(2q)$ (the set $B_j \cap B(x_j,3/(2q))^c$), we have:
$$\eta_{\sigma}(x)=\eta_{\sigma}(\tilde{x})=1/2.
$$
Now, if $\tilde{x}$ belongs to $B(x_j,3/(2q))$:
$$|\eta_{\sigma}(x)-\eta_{\sigma}(\tilde{x})| = \left|\frac{\Phi_j(\tilde{x})}{2}\right| \leq \frac{c_{\varphi} \|\varphi\|_{\infty}}{2q} \leq \frac{c_{\varphi} \|\varphi\|_{\infty}}{} \|x-\tilde{x}\|
$$
Hence, we can deduce the uniform Lipschitz bound (note that the case $x\in B_j$ and $\tilde{x}\in B_k $ can be treated in the same way):
$$\forall (x,\tilde{x})\in (\RR^{d})^{2}, \forall \sigma\in\Sigma_{n} \qquad 
|\eta_{\sigma}(x)-\eta_{\sigma}(\tilde{x})|\leq c_{\varphi} \frac{\|\varphi'\|_{\infty}+\|\varphi\|_{\infty}}{} \|x-\tilde{x}\|.
$$

\paragraph{Minoration of the risk}
Following the arguments of Theorem 3.5 in \cite{AudTsy}, we have that:
$$
\mathcal{R}_{n}\geq m\frac{\omega}{q}\left(1-q^{-1}\sqrt{n\omega}\right).
$$

\subsubsection{Proof of Theorem \ref{theo:NON2} and of Theorem \ref{theo:Gene1}, i)\label{sec:proofNON2}}
We first study the situation of rates when $\blabla$ and $\PTf$ are in force and use a measure similar to the one represented on the left of Figure \ref{fig:mesure}.
\paragraph{Besicovitch-like condition $\blabla$:}
We aim to show that our network satisfies the lower bound involved in $\blabla$. Consider $\delta \rightarrow 0^+$ and $q\rightarrow + \infty$. 
If $x\in A_{0}$ and $\delta = o(1/q)$ then one ball $B_j$  intersects at the least half of $B(x,\delta)$ and since $\mu$ is stepwise constant:
$$
\mathbb{P}_X\left(B(x,\delta)\right)\geq\frac{\mu(x)\lambda\left(B(x,\delta)\right)}{2}\geq \frac{\gamma_d}{2} \mu(x)\delta^{d}
$$
If $\delta$ is now proportional to $1/q$,  the last inequality is still true up to a constant (which is not illustrated here for the sake of simplicity).
Now if $q^{-1}=o(\delta)$, $B(x,\delta)$ contains a number $N_{\delta,q}$ of balls $(B_j)_{1 \leq j \leq m}$ such that
$
N_{\delta,q} \geq C_d \frac{\delta^d}{q^{-d}}.
$
In this case, we still have
$$
\mathbb{P}_X\left(B(x,\delta)\right)\geq\mathbb{P}_X\left(B(x,\delta)\cap \cup_{j=1}^m B_j \right) \geq N_{\delta,q} \times  \omega \geq C_d \delta^d q^d \omega = \frac{C_d}{\gamma_d} \mu(x) \delta^d.
$$ 
Hence, the measure $\mathbb{P}_X$  belongs to $\blabla$ with a constant $\kappa$ independent of $q$.

\paragraph{Tail Assumption $\PT$ or $\PTf$}
First, note that $\mathbb{P}_X$ is built such that if $x\in A_0$
$$
\mathbb{P}_X(\mu<\epsilon)=0\ \mathrm{if}\ \epsilon<\omega q^{d}/(\gamma_{d}2^d)\ \mathrm{and}\ \mathbb{P}_X(\mu<\epsilon) = m \omega\ \mathrm{if}\ \epsilon>\omega q^{d}/(\gamma_{d}2^d).$$
 Note that the density on $A_1$ is bounded from below and, as a result, we will not take the tail property on this set into account.

Since $\psi$ is increasing in a neighborhood of $0$, the tail property $\mathbb{P}_X(\mu<\epsilon) \lesssim \psi(\epsilon)$ is fulfilled as soon as:
$$
m \omega \lesssim \psi\left(\frac{\omega q^{d}}{\gamma_d2^d}\right).
$$

\paragraph{Calibration for the minoration}
Recall that 
$
\mathcal{R}_{n}\geq m\frac{\omega}{q}\left(1-q^{-1}\sqrt{n\omega}\right)
$
and that we must satisfy the following constraints 
$$
m\omega=O(q^{-\alpha})\ \mathrm{and}\,
m \omega \lesssim \psi\left(\frac{\omega q^{d}}{\gamma_d2^d}\right).
$$
The lower bound above is meaningful as soon as we choose $\omega \leq \frac{q^{2}}n$.
If we denote $\epsilon_{n,\alpha,d}=q^{-1}$, the values of $m,q,\omega$ that provide a tradeoff between all these constraints are obtained with

$$
m \omega=q^{-\alpha}, \, \frac{\omega q^d}{\gamma_d2^d} = \psi^{-1}(m\omega),\, \omega=\frac{q^2}{2n}.
$$
In particular, the constraints are optimized when $\epsilon_{n,\alpha,d}$ solves 
$2^{-d}\gamma_d^{-1} \frac{\epsilon_{n,\alpha,d}^{-2}}{2n} \epsilon_{n,\alpha,d}^{-d} = \psi^{-1}(\epsilon_{n,\alpha,d}^{\alpha}),$
 which leads to the lower bound
$$ \R_n \gtrsim  \epsilon_{n,\alpha,d}^{1+\alpha}\quad  \text{with} \quad  n^{-1}= \epsilon_{n,\alpha,d}^{d+2} \psi^{-1}(\epsilon_{n,\alpha,d}^{\alpha}).$$

In the above calibration, we obtain that:
$$
\omega=q^{-d} \psi^{-1}(q^{-\alpha}) \quad \text{and} \quad m=q^{d} \frac{q^{-\alpha}}{\psi^{-1}(q^{-\alpha})}.
$$
This ends the proof of Theorem \ref{theo:NON2} and Theorem \ref{theo:Gene1}, i).
\hfill $\square$

Looking carefully at the proof of the theorem above, we can see that the influence of $\psi$ is as follows:
\begin{itemize}
\item If $\epsilon = o(\psi(\epsilon))$, then the construction of the network yields a non compactly supported distribution since: $$\lambda(Supp(\mu)) \gtrsim  m q^{-d} = \frac{q^{-\alpha}}{\psi^{-1}(q^{-\alpha})} \longrightarrow + \infty \qquad \text{as} \qquad 
 q \longrightarrow + \infty.
 $$
 As pointed out in paragraph \ref{sec:examples}, a polynomial decay of the density when $x$ grows to $\infty$ yields such a tail size.
\item In the opposite situation, when $\psi(\epsilon)= O(\epsilon)$, the corresponding density has a compact support. In particular, when $\psi(\epsilon) \sim \epsilon$, our network is exactly the same as the one used in  \cite{AudTsy} and we naturally recover the lower bound $n^{-(1+\alpha)/(2+\alpha+d)}$.
\end{itemize}


\subsubsection{Proof of Theorem \ref{theo:NON1}, item $i)$}\label{sec:mmasanstail}

We study the specific case where the Besicovitch has to be fulfilled although the Tail Assumption is no longer necessary. In such a case, we still use the construction shown on the left of Figure \ref{fig:mesure} and provided in Section \ref{sec:proofNON2} but \textit{$m$ can be chosen much greater than $q^d$}. For example, for a parameter $\tau>0$ chosen in the sequel, we assume that $m=q^{d+\tau} >> q^{d}$ as $q \longrightarrow + \infty$. In such a case, the underlying measure $\mathbb{P}_X$ is no longer compactly supported. 

Using the same argument as above, Assumption \textbf{A3} is still satisfied since the number $m$ of balls $B_i$ does not influence the minoration of $\mathbb{P}_X(B(x,\delta))$.

We have to satisfy the following constraints:
$$
m \omega = O(q^{-\alpha}), \omega \leq \frac{q^2}{n}.
$$
We keep the value of $\omega$ as:
$$
\omega = q^2/(2n),
$$
and the calibration of $q$ with respect to $m$ yields
$$
q=n^{\frac{1}{2+d+\alpha+\tau}}.
$$
We then obtain the lower bound
$$ \R_n \geq c_{\phi} n^{-\frac{1+\alpha}{2+\alpha+d+\tau}}.$$
By increasing the size of $\tau$ ($\tau_n = n$ for example), it can then be observed that it is possible to obtain any arbitrary value between $0$ and $c_\phi$. Hence, for any classifier $\Phi_n$, a distribution on $(X,Y)$ exists
such that Assumptions \textbf{A1-A3} hold
 and that the classifier $\Phi_n$ cannot be consistent.
\subsubsection{Proof of Theorem \ref{theo:NON1}, item $ii)$}\label{sec:sansmma}
We then study what could happen when the Tail Assumption is satisfied but Assumption \textbf{A3} can be violated. 
The idea is to pick the density of observations to ensure the validity of the Tail Assumption. To do this, we consider the new marginal on $X$ whose $\mu$ defined as:
$$
\forall x \in B_j \qquad \mu(x) =  \omega \frac{q^{\gamma d} \left( 1-|x-x_j| q^{\gamma}\right)_{+}}{\int_{B(0,1)} \left( 1-|x| \right)_{+} dx}
$$
so that:
$$
\int_{B_j} \mu(x) dx = \omega.
$$
The obtained measure is represented on the right of Figure \ref{fig:mesure}.
We proceed in the same way as in paragraph \ref{sec:baseline}:  $\varphi$ is still lower bounded by a strictly positive constant (as soon as $\gamma \geq 1$) and the Margin Assumption is satisfied as soon as $m \omega = O(q^{-\alpha})$.

It should also be observed that Assumption \textbf{A3} is not satisfied here. In fact, when we choose $\gamma >1$ and the reference radius $\delta$  as $\delta = q^{-a}$ for $a \in [1,\gamma[$:

$$
\mathbb{P}_X(B(x_j,q^{-a})) = \omega  \quad \text{and} \quad \delta^d \mu(x_j) =c q^{- a d} \omega q^{\gamma d} = c \omega q^{d(\gamma -a)},
$$
where 
$$
c=\frac{1}{\int_{B(0,1)} \left( 1-|x| \right)_{+} dx}
$$
The left hand side becomes negligible with respect to the right hand side as soon as $q \longrightarrow + \infty$.

We now check that such a definition of density $\mu$ satisfies the Tail Assumption. Consider any $\epsilon>0$. We then have:
\begin{eqnarray*}
\lefteqn{\mathbb{P}_X \left( \{\mu < \epsilon\} \right)}\\
 &=& m \omega \int_{B(x_1,1/q)} c  q^{\gamma d}   \left( 1-|x-x_1| q^{\gamma}\right)_{+} \mathbf{1}_{\{c \omega q^{\gamma d} \left( 1-|x-x_j| q^{\gamma}\right)_{+} \leq \epsilon\}} dx \\
  &=&   m \omega \int_{B(x_1,1/q)} c  q^{\gamma d}   \left( 1-|x-x_1| q^{\gamma}\right)_{+}
  \mathbf{1}_{\{  \left( 1-|x-x_j| q^{\gamma}\right)_{+} \leq c^{-1} \omega^{-1} q^{- \gamma d} \epsilon\}} dx. 
\end{eqnarray*}
Consider the variable $y=q^{\gamma} ( x - x_1)$. We then obtain
$$
\mathbb{P}_X\left( \{\mu < \epsilon\} \right)  = m \omega \int_{B(0,1)} c   \left( 1- |y| \right)_{+}  \mathbf{1}_{\{  \left( 1-|y|\right)_{+} \leq c^{-1} \omega^{-1} q^{- \gamma d} \epsilon\}} dy \leq  \gamma_d  m q^{-\gamma d} \epsilon.
$$
As a consequence, the Tail Assumption is true as soon as $m = O(q^{\gamma d})$. We point out that since we chose $\gamma>1$  in the sequel, $m$ is then greater than $q^{d}$ and the support of $\mu$ is no longer compact since $q \longrightarrow + \infty$.

Following the roadmap of paragraph \ref{sec:proofNON2}, we then obtain the lower bound calibrations of $q$ and $\omega$ such that:
$$
\R_n \geq n^{-\frac{1+\alpha}{2+\alpha+\gamma d}}.
$$
Again, a sufficiently large value of $\gamma $ makes it possible to obtain arbitrarily slow rates (and even non-consistent classifiers).

\subsection{Proof of Theorem \ref{theo:Main1}}
\label{preuve_dim_finie}

Let $\epsilon>0$ be a given real number (whose value will be specified later), and define:
$$
\Be := \left\{x \in \mathbb{R}^d \quad \vert \quad |\eta(x)-1/2| \leq \epsilon \right\}.
$$

Applying Proposition \ref{prop:gyorfi} in Section \ref{s:technical}, the excess risk can be decomposed as follows:

\begin{eqnarray*}
\mathcal{R}(\Phi_n) - \mathcal{R}(\Phi^*) &=& \EE \left[ |2 \eta(X)-1| \mathbf{1}_{\{\Phi_n(X) \neq \Phi^*(X)\}}\right], \\
& = &  \underbrace{\EE \left[ |2 \eta(X)-1| \mathbf{1}_{\{\Phi_n(X) \neq \Phi^*(X)\}}  \mathbf{1}_{X \in \Be} \right]}_{:=T_{1,\epsilon}} \\
& & +
\underbrace{ \EE \left[ |2 \eta(X)-1| \mathbf{1}_{\{\Phi_n(X) \neq \Phi^*(X)\}}\mathbf{1}_{X \in \Be^c} \right] }_{:=T_{2,\epsilon}}.
\end{eqnarray*}
Now, the Margin Assumption \textbf{A2} yields:
\begin{equation}\label{eq:control1}
T_{1,\epsilon} \leq 2 \EE \left[ |\eta(X)-1/2|  \mathbf{1}_{X \in \Be} \right] \leq 2 \epsilon \PP_X(X \in\Be) \leq 2 C \epsilon^{1+\alpha}.
\end{equation}
In order to control $T_{2,\epsilon}$, define:
$$
\forall j \geq 1 \qquad 
\Bej := \left\{ x \in \mathbb{R}^d \quad \vert \quad 2^{j-1} \epsilon \leq  |\eta(x)-1/2| \leq 2^j \epsilon \right\}.
$$
Now, 
\begin{eqnarray*}
T_{2,\epsilon} & =& 2 \sum_{j \geq 1} \EE \left[  |\eta(X)-1/2|  \mathbf{1}_{\{\Phi_n(X) \neq \Phi^*(X)\}}  \mathbf{1}_{\{X \in \Bej \}} \right]\\
& \leq & 2 \epsilon  \sum_{j \geq 1}2^{j} \EE_X \left[  \mathbf{1}_{\{X \in \Bej \}} \EE_{\otimes^n} \left( \mathbf{1}_{\{\Phi_n(X) \neq \Phi^*(X)\}} \right) \right].
\end{eqnarray*}
We can apply Proposition \ref{prop:use} (see Section \ref{s:technical} below)  to obtain:
\begin{equation}\label{eq:serie}
T_{2,\epsilon} \leq 4 \epsilon \sum_{j \geq 1}2 ^{j}
\EE_X \left[  \mathbf{1}_{\{X \in \Bej \}}
 \exp \left(- 2 k_n \lfloor 2^{j-1} \epsilon - \Delta_n(X)  \rfloor_{+}^2 \right) \right].
\end{equation}
Since $\mu$ is lower bounded by $a>0$ on $\Omega$, we can apply Proposition \ref{prop:Delta2} with $a=\mu_-$ to obtain:
$$
\Delta_n(X)  \leq C \left( \left(\frac{k_n}{n} \mu_-^{-1}\right)^{1/d} + \exp\left(-3k_{n}/14\right)\right).
$$
 Now, we consider  $\epsilon=\epsilon_n \geq 2 \Delta_n(X)$ , for example by choosing:
\begin{equation}\label{eq:choose_eps}
\epsilon_n := 2C \left(  \left(\frac{k_n}{n} a^{-1} \right)^{1/d} + \exp\left(-3k_{n}/14\right)\right) .
\end{equation}
With $\epsilon_n$ defined as in (\ref{eq:choose_eps}), we deduce that 
$ 2^{j-1} \epsilon_n - \Delta_n(X) \geq 2^{j-1} \epsilon_n - \frac{\epsilon_n}{2} \geq \epsilon_n \left( 2^{j-1} - \frac{1}{2}\right) >0$. Thus, \eqref{eq:serie} becomes:
 \begin{eqnarray*}
T_{2,\epsilon_n}& \leq & 4 \epsilon_n \sum_{j \geq 1}2^{j} \EE_X \left[\mathbf{1}_{\{ 0 < |\eta(X)-1/2| < 2^{j} \epsilon_n\}} \exp\left(-2 k_{n}\epsilon_n^{2} \left(2^{j-1}-1/2\right)^{2}\right)\right].  
 \end{eqnarray*}
 Now, in order to control the previous bound, we choose $k_n$ such that:
\begin{equation}\label{eq:choose_k} 
 k_n = \epsilon_n^{-2}.
 \end{equation}
 Thanks to \eqref{eq:choose_eps}, the constraint \eqref{eq:choose_k} then yields:
 \begin{equation}\label{eq:taille_eps}
 \epsilon_n \sim n^{\frac{-1}{2+d}} \quad \text{and} \quad k_n  \sim n^{\frac{2}{2+d}}.
 \end{equation}
 We then obtain  that:
 \begin{eqnarray*}
T_{2,\epsilon_n} & \leq & 4 \epsilon_n \sum_{j \geq 1}2^{j} \EE_X \left[\mathbf{1}_{\{ 0 < |\eta(X)-1/2| < 2^{j} \epsilon_n \}} \exp\left(-\frac{2^{2j}}{8} \right)\right],\\
 & \leq &\epsilon_n  \sum_{j \geq 1} 2^{j+2}  \exp\left(-\frac{2^{2j}}{8}\right) \mathbb{P}_X \left(  |\eta(X)-1/2| < 2^{j} \epsilon_n\right) .  
 \end{eqnarray*}
The Margin Assumption applied to $\mathbb{P}_X \left(  |\eta(X)-1/2| < 2^{j} \epsilon_n\right)$ leads to:
 \begin{eqnarray*}
T_{2,\epsilon_n} & \leq & \epsilon_n^{1+\alpha} \sum_{j \geq  1} 2^{j(1+\alpha)+2}  \exp\left(-\frac{2^{2j}}{8}\right).
 \end{eqnarray*}
The series on the right hand side converges. This last bound associated with \eqref{eq:control1} leads to:
 $$
 \sup_{F\in\mathcal{F}} \left[ \mathcal{R}(\Phi_n)- \mathcal{R}(\Phi^{*}) \right] \leq C n^{-\frac{1+\alpha}{2+d}} .$$

\subsection{Proof of the upper bounds: Theorem \ref{theo:Main2} and Theorem \ref{theo:Gene1} ii)}
\label{s:preuve_noncompact}

\begin{proof}[Proof of Theorem \ref{theo:Main2}]
We consider a constant $\gamma$ and  use the following decomposition of $\RR^{d}$ for a suitable $\gamma>0$ (that will be chosen later on):
$$
\RR^{d}= \underbrace{\{x : 0\leq \mu(x) \leq n^{-\gamma} \}}_{R_{n}}
\cup \underbrace{\{x : \mu>n^{-\gamma}\}}_{Q_{n}}.$$

We follow the roadmap of the proof of Theorem \ref{theo:Main1} and keep the notation $\Be$, which refers to $\Be := \left\{x \in \mathbb{R}^d \,  : \, |\eta(x)-1/2| \leq \epsilon \right\}$. Thanks to Proposition \ref{prop:gyorfi}, we obtain:
\begin{eqnarray*}
\mathcal{R}(\Phi_n) - \mathcal{R}(\Phi^*) &=& \EE \left[ |2 \eta(X)-1| \mathbf{1}_{\{\Phi_n(X) \neq \Phi^*(X)\}}\right] \\
& = &   \underbrace{ \EE \left[ |2 \eta(X)-1| \mathbf{1}_{\{\Phi_n(X) \neq \Phi^*(X)\}}\mathbf{1}_{X \in R_n} \right] }_{:=T_{R_n}}  \\ 
& &+ \underbrace{ \EE \left[ |2 \eta(X)-1| \mathbf{1}_{\{\Phi_n(X) \neq \Phi^*(X)\}}\mathbf{1}_{X \in Q_{n}} \right] }_{:=T_{Q_{n}}}. \\ 
\end{eqnarray*}

\paragraph{Study of $R_n$}
The Tail Assumption \textbf{A4} in the particular case where $\psi = Id$ leads to:
\begin{align*}
T_{R_n}&\leq\PP_X\left(X \in R_{n}\right)= \PP_X (\mu(X) \leq n^{-\gamma}) \lesssim n^{-\gamma}.
\end{align*}

\vspace{0.5cm}

\paragraph{Study of $Q_n$} Following the proof of Theorem \ref{theo:Main1} with $a=n^{-\gamma}$, Equations \eqref{eq:serie}-\eqref{eq:choose_k} yield:
\begin{equation}
T_{Q_n} \leq C \epsilon_n^{1+\alpha},
\label{eq:TQn}
\end{equation}
where $\epsilon_n$ and $k_n$ satisfy the balance equations
$$ \epsilon_n \sim  2C  \left(\frac{k_n}{n} a^{-1} \right)^{1/d}  = 2C  \left(\frac{k_n}{n^{1-\gamma}} \right)^{1/d} \qquad \text{and} \qquad  k_n = \epsilon_n^{-2}.$$
The equilibria are met in the two terms above with
\begin{equation}
k_{n} \sim C n^{\frac{2(1-\gamma)}{2+d}},\qquad \text{and} \qquad 
\epsilon_{n}\lesssim n^{-\frac{(1-\gamma)}{2+d}}.
\label{eq:TQn2}
\end{equation}

\paragraph{Final control of the risk.} From the previous bounds, we obtain that:
\begin{equation}\label{eq:finalcontrol} \mathcal{R}(\Phi_n) - \mathcal{R}(\Phi^*) \lesssim n^{-\frac{(1-\gamma)(1+\alpha)}{2+d}}  +  n^{-\gamma}.\end{equation}

We optimize the last expression with respect to $\gamma$ by setting
$$(1-\gamma) (1+\alpha) = \gamma(2+d) \Leftrightarrow \gamma = \frac{1+\alpha}{3+\alpha+d}.$$
The above choices allow us  to conclude that:
 $$
 \sup_{F\in\mathcal{F}} \left[ \mathcal{R}(\Phi_n)- \mathcal{R}(\Phi^{*}) \right] \leq C  n^{-\frac{1+\alpha}{3+\alpha+d}}.$$
 \end{proof}

\begin{proof}[Proof of Theorem \ref{theo:Gene1}] ii)
 We follow the roadmap of the previous proof and replace the threshold $n^{-\gamma}$ with $a_n$, which should be carefully chosen.
  The key balance is still 
$k_n = \nu_n^{-2}$ on the set $\left\{ \mu \geq a_n \right\}$ with the optimal setting:
$$
 \frac{k_n}{n a_n} \lesssim \nu_n^{d} 
$$
Since we want to obtain a minimal value for $\nu_n$, this last equation leads to the choice:
\begin{equation}\label{eq:choose_thresh}
a_n = \frac{1}{n \nu_n^{2+d}},
\end{equation}
and the upper bound of the excess risk we obtained is then
 $$
 \sup_{F\in\mathcal{F}} \left[ \mathcal{R}(\Phi_n)- \mathcal{R}(\Phi^{*}) \right] \lesssim \nu_n^{1+\alpha} + \psi(a_n).
 $$
 The natural equilibrium is found when plug-in \eqref{eq:choose_thresh} in this last upper bound and $\nu_n$ are be fixed so that  $ \psi^{-1} \left( \nu_n^{1+\alpha} \right) = n^{-1} \nu_n^{-(2+d)}$. 
We then obtain the rate $\nu_n^{1+\alpha}$ with the balance equation:
$$
\nu_n^{2+d}  \psi^{-1}(\nu_n^{1+\alpha}) \sim n^{-1}.
$$
\end{proof}

\subsection{Proof of Theorem \ref{theo:Main_opt} (sliced nearest neighbor)}
\begin{proof}
We use the partition of $\Omega$ naturally derived from  the slices $\Omega_{n,0}$ and $(\Omega_{n,j})_{j \geq 1}$:
$$
\Omega_{n,0} := \left\{ x \vert \mu(x) \geq n^{- \gamma} \right\} \quad \text{and} \quad 
\Omega_{n,j} := \left\{ x \vert n^{- \gamma} 2^{-(j+1)} \leq  \mu(x) \leq n^{- \gamma} 2^{-j} \right\} .
$$
For this purpose, let  $\gamma \in (0,1)$ (that will be specified later) and:
$$k_{n,j}=k_{n,0} 2^{-2j/(2+d)} \quad \text{with} \quad k_{n,0} = n^{\frac{2(1-\gamma)}{2+d}} \log(n).$$
We then use the following decomposition of the excess risk:
\begin{eqnarray*}
\lefteqn{\mathcal{R}(\Phi_n) - \mathcal{R}(\Phi^*)}\\
 &=& \EE \left[ |2 \eta(X)-1| \mathbf{1}_{\{\Phi_n(X) \neq \Phi^*(X)\}}
\left( \mathbf{1}_{\Omega_{n,0}} + \sum_{j=1}^{+ \infty} \mathbf{1}_{\Omega_{n,j}}\right)\right],\\
& = & \EE [ |2 \eta(X)-1| \mathbf{1}_{\{\Phi_n(X) \neq \Phi^*(X)\}} \mathbf{1}_{\Omega_{n,0}}]\\
&& +
\sum_{j=1}^{+ \infty} \EE[ |2 \eta(X)-1| \mathbf{1}_{\{\Phi_n(X) \neq \Phi^*(X)\}} \mathbf{1}_{\Omega_{n,j}},\\
& := & T_n +  \sum_{j=1}^{+\infty} R_{n,j}.
\end{eqnarray*}
\paragraph{Study of  $T_n$} The density is lower bounded by $n^{-\gamma}$  on $\Omega_{n,0}$. The proof of Theorem \ref{theo:Main2} yields (see (\ref{eq:TQn}) and (\ref{eq:TQn2})):
$$
T_n \leq n^{-(1+\alpha)\frac{1-\gamma}{2+d}}.
$$
\paragraph{Study of $R_{n,j}$} For any $j > J_0(n):=(1-\gamma) \frac{\log(n)}{\log(2)}$ and 
for any $x \in \Omega_{n,j}$ with $j>J_0 $, we have:
$$\mu(x) < n^{-\gamma} 2^{-(1-\gamma) \frac{\log(n)}{\log(2)}} =  1/n.$$ 
The Tail Assumption with $\psi=Id$ leads to:
$$
\sum_{j > J_0} R_{n,j} \leq \PP_X \left[ \mu(X) \leq n^{-1} \right] \lesssim n^{-1}.
$$
\paragraph{Study of $R_{n,j}, j \leq J_0(n)$} 
We consider the intermediary slices, and for $1 \leq j \leq J_0(n)$:
\begin{eqnarray*}
 R_{n,j} &=& \overbrace{\EE \left[ |2 \eta(X)-1| \mathbf{1}_{\{\Phi_n(X) \neq \Phi^*(X)\}}
 \left[ \mathbf{1}_{|\eta(X)-1/2| \leq \epsilon_{n,j}}\right]
  \mathbf{1}_{X \in \Omega_{n,j}}\right]}^{:=R_{n,j,1}} \\
  & & +  \underbrace{\EE \left[ |2 \eta(X)-1| \mathbf{1}_{\{\Phi_n(X) \neq \Phi^*(X)\}}
 \left[ \mathbf{1}_{|\eta(X)-1/2| > \epsilon_{n,j}}\right]
  \mathbf{1}_{X \in \Omega_{n,j}}\right]}_{:=R_{n,j,2}} 
\end{eqnarray*}
where $\epsilon_{n,j}$ will be chosen later.
To bound $R_{n,j,1}$, we use the fact that $|\eta-1/2| \leq \epsilon_{n,j}$ as well as the Tail Assumption
on the set $\Omega_{n,j} \subset \left\{
\mu(X) \leq n^{-\gamma} 2^{-j} \right\}$ to obtain:

\begin{equation}\label{eq:rnj1}
 R_{n,j,1} \leq 2 \epsilon_{n,j}  n^{-\gamma} 2^{-j}.
\end{equation}
Thanks to Proposition \ref{prop:use}, we can bound the term $R_{n,j,2}$ as follows:

$$
R_{n,j,2} \leq 4 \EE\left[ \mathbf{1}_{X \in \Omega_{n,j}} \exp \left(- 2 k_{n,j} {\lfloor \epsilon_{n,j} - \Delta_n(X)  \rfloor_{+}}^2 \right)\right].
$$
The term $\epsilon_{n,j}$ is then chosen such that $\epsilon_{n,j}-  \Delta_n(X) \leq \epsilon_{n,j}/2$. According to Proposition \ref{prop:Delta2}, we obtain:
 $$
 \epsilon_{n,j} = c \left(k_{n,j} \frac{n^{\gamma} 2^{j+1}}{n} \right)^{1/d},
 $$
where $c$ is chosen large enough. With this value, we obtain the following simplifications:
 
\begin{equation}\label{eq:rnj2}
\epsilon_{n,j}^2 k_{n,j} = c^2 \frac{k_{n,j}^{1+\frac{2}{d}} 2^{\frac{j+1}{d}}}{n^{\frac{1-\gamma}{d}}}
= c^2 
 \frac{k_{n,0}^{1+\frac{2}{d}}}{n^{\frac{1-\gamma}{d}}} 2^{- \frac{2j}{2+d} \left( 1+\frac{2}{d} \right)} 2^{(2j+2)/d} = c^2 2^{2/d}\log(n)^{1+2/d}.
\end{equation}
Taken together, \eqref{eq:rnj1} and \eqref{eq:rnj2} lead to:
$$
R_{n,j} \leq 2 \epsilon_{n,j} n^{-\gamma} 2^{-j} + 4  \exp(- c^2 2^{2/d}\log(n)^{1+2/d}) \PP(X \in \Omega_{n,j}).
$$
We then sum up all these terms for $j  \geq 1$:
\begin{eqnarray*}
R_n &\lesssim& \sum_{j =1}^{J_0} \frac{\log(n)^{1/2+1/d}}{\sqrt{k_{n,j}}} n^{-\gamma} 2 ^{ -j} +\frac{1}{n} \sum_{j=1}^{J_0} \PP(X \in \Omega_{n,j}), \\
& \lesssim & n^{-\gamma}\log(n)^{1/2+1/d}\sum_{j=1}^{J_0} 2^{-j} k_{n,j}^{-1/2}  +\frac{1}{n}, \\
& \lesssim & n^{-\gamma} n^{-\frac{1-\gamma}{2+d}}\log(n)^{1/2+1/d}\sum_{j=1}^{J_0} 2^{-j+\frac{j+1}{2+d}}  +\frac{1}{n}, \\
& \lesssim & n^{-\gamma} n^{-\frac{1-\gamma}{2+d}}\log(n)^{1/2+1/d}+\frac{1}{n}.
\end{eqnarray*}

We can see in this last upper bound that we obtain an improvement between the standard rule and the one fixed here since the term $n^{-\gamma}$ that appears in the tail of $\mu$ on the right hand side of \eqref{eq:finalcontrol} is transformed into $n^{-\gamma} \times n^{-\frac{1-\gamma}{2+d}}$ up to a log term.

\paragraph{Final equilibrium} We now fix the optimal value of $\gamma$ with the conjunction of the upper bounds for $R_n$ and $T_n$:
$$
\mathcal{R}(\Phi_n) - \mathcal{R}(\Phi^*)  \lesssim n^{-(1+\alpha)\frac{1-\gamma}{2+d}} + n^{-\gamma -\frac{1-\gamma}{2+d}}\log(n)^{1/2+1/d}+O(1/n).
$$
The balance equilibrium is reached with $(1+\alpha)\frac{1-\gamma}{2+d} = \gamma +\frac{1-\gamma}{2+d}$,
meaning that
$\gamma=\frac{1}{2+\alpha+d}$. This concludes the proof. \end{proof}

\subsection{Technical results}
\label{s:technical}

In the following, we use the result reported in \cite{Gyorfi78} that compares the excess risk of any classifier with the Bayes procedure.
\begin{prop}[\cite{Gyorfi78}]\label{prop:gyorfi}
For any classifier $\Psi$, we have:
$$
\mathcal{R}(\Psi) - \mathcal{R}(\Phi^*) = \EE \left[ |2 \eta(X)-1| \mathbf{1}_{\{\Psi(X) \neq \Phi^*(X)\}}\right].
$$
\end{prop}

\vspace{1cm}

The following lemma is concerned with the concentration of the plug-in estimator $\hat{\eta}_n$ (see \eqref{eq:def_eta_n}). 

 \begin{lemma}[Concentration of $\hat{\eta}_n$] \label{lem:concentration} In the classification model, 
$$ \PP_{\otimes^n} \left( |\hat{\eta}_n(X) - \EE_{\otimes^n} \left(\eta_n(X)\right)|>s\right)\leq 2 \exp\left(-2k_{n}s^{2}\right).$$
\end{lemma}

\begin{proof}
Once again, we can observe that conditionally to the $(X_{(i)})_{1 \leq i \leq n}$, the corresponding labels $(Y_{(i)}(X))_{1 \leq i \leq n}$ are \textit{independent} Bernoulli random variables with respective parameters $\eta(X_{(i)})$. We can now use the Hoeffding inequality as follows:
\begin{eqnarray*}
\lefteqn{\PP_{\otimes^n} \left( |\hat{\eta}_n - \EE_{\otimes^n} \left(\eta_n(X)\right)|>s\right)} \\
&=& \EE_{\otimes^n} \left( \PP_{\otimes^n} \left( |\eta_n(X) - \EE_{\otimes^n} \left(\eta_n(X)\right)|>s \vert (X_1,\ldots, X_n) \right) \right). \\
& \leq  & \EE_{\otimes^n} \left( \PP_{\otimes^n}\left(\left|  \frac{1}{k_n} \left[ \sum_{i=1}^{k_n} Y_{(i)}(X) - \eta(X_{(i)}) \right]\right|> s
\vert (X_1,\ldots, X_n) \right)  \right). \\
& \leq &  \EE_{\otimes^n} \left( 2 \exp(- 2 k_n s^2) \vert (X_1,\ldots, X_n)  \right) \leq 2  \exp(- 2 k_n s^2).
\end{eqnarray*}
\end{proof}

\vspace{0.5cm}

We first state an important upper bound of the error rate when the design point $X$ is fixed.

\begin{prop}\label{prop:use}
For any $\epsilon>0$ and any $X \in \Omega$, if $\Delta_n(X):=|\mathbb{E}_{\otimes^n}\hat{\eta}_n(X) - \eta(X)|$, we have:
$$
\mathbf{1}_{\{|\eta(X) - \frac{1}{2}| \geq \epsilon\}} \,
\EE_{\otimes^n} \left[ 
\mathbf{1}_{\{\Phi_n(X) \neq \Phi^*(X)\}} 
 \right] \leq 2\mathbf{1}_{\{|\eta(X) - \frac{1}{2}| \geq \epsilon\}}  \,
 e^{- 2 k_n{\lfloor \epsilon - \Delta_n(X)\rfloor_{+}}^2 },
$$
where  $\lfloor a \rfloor_{+}$ refers to the positive part of any real number $a$.
\end{prop}

\begin{proof}
In the event $ \left\lbrace \Phi_n(X) \not = \Phi^*(X)  \right\rbrace$, $\hat{\eta}_n(X)-1/2$ and $\eta(X)-1/2$ do not have the same sign. Therefore:
$$
\mathbf{1}_{\{|\eta(X) - \frac{1}{2}| \geq \epsilon\}} 
 \mathbf{1}_{\{\Phi_n(X) \neq \Phi^*(X)\}}    \leq  \mathbf{1}_{\{   |\eta(X)-1/2| \geq \epsilon\}} 
  \mathbf{1}_{ \{ |\hat{\eta}_n(X) - \eta(X)| \geq  \epsilon\}}.
$$
Then:
\begin{eqnarray*}
\lefteqn{\EE_{\otimes^n} \left[ 
\mathbf{1}_{\{\Phi_n(X) \neq \Phi^*(X)\}}\mathbf{1}_{\{|\eta(X) - \frac{1}{2}| \geq \epsilon\}} 
 \right]}\\
 &\leq & \EE_{\otimes^n} \left[\mathbf{1}_{\{  |\eta(X)-1/2| \geq \epsilon\}} 
  \mathbf{1}_{ \{ |\hat{\eta}_n(X) - \eta(X)| \geq  \epsilon\}}  \right] \\ 
  & = & \mathbf{1}_{\{|\eta(X) - \frac{1}{2}| \geq \epsilon\}}  \,  \PP_{\otimes^n} \left( |\hat{\eta}_n(X) - \eta(X)| \geq \epsilon \right) \\
  & \leq & \mathbf{1}_{\{|\eta(X) - \frac{1}{2}| \geq \epsilon\}} \,  \PP_{\otimes^n} \left( |\hat{\eta}_n(X) - \EE_{\otimes^n} \left(\hat{\eta}_n(X)\right)| \geq \epsilon - \Delta_n(X)  \right)  \\
\end{eqnarray*}
where  the last line follows from the triangular inequality and the definition of $\Delta_n(X)$.
Lemma \ref{lem:concentration} applied with $s= \lfloor \epsilon - \Delta_n(X) \rfloor_{+}$ now leads to the conclusion of the proof.
\end{proof}

\vspace{1cm}

This last proposition clearly underlines the effect of the bias term $\Delta_n(X)$ in the misclassification error rate. This bias term is then upper bounded by the next result.

\begin{prop}\label{prop:Delta2}
Assume that $\mu$ belongs to $\blabla$ and $\eta$ belongs to $\mathcal{C}^{1,0}(\Omega,L)$.
Then, a constant $C>0$ exists such that for any $a>0$:
$$
\left|\EE_{\otimes^n} \hat{\eta}_n (x) - \eta(x) \right| \leq L \left(\frac{2}{\kappa}\right)^{1/d}\left( \frac{ k_n}{n a } \right)^{1/d} +  2\exp \left(- \frac{3k_{n}}{14} \right),
$$
\color{black}
for all $x\in K_a$ where
$$K_{a}:=\left\{ x \in \RR^d \, \vert \, \mu(x) \geq a  \right\}.$$
\end{prop}

\begin{proof}  We first propose a control of the bias and then use a concentration inequality in order to obtain the bound. 
\paragraph{Decomposition of the bias} Let $x\in K_a$ be fixed. According to the definition of $\hat{\eta}_n(x)$ (see \eqref{eq:def_eta_n}):
$$ \EE_{\otimes^n} [\hat{\eta}_n(x)] =  \EE_{\otimes^n} \left[ \frac{1}{k_n} \sum_{j=1}^{k_n} Y_{(j)}(x) \right] 
= \EE_{\otimes^n} \left[ \frac{1}{k_n} \sum_{j=1}^{k_n} \eta(X_{(j)}) \right]. $$
Hence:
$$
\Delta_n(x)  =  \left| \EE_{\otimes^n}  [\hat{\eta}_n (x)] - \eta(x) \right|  = \left|  
 \EE_{\otimes^n} \left( \frac{1}{k_n} \sum_{i=1}^{k_n} \eta(X_{(i)})- \eta(X) \right) \right|.
$$
For any $t \geq 0$, we write:
$$
\Delta_n(x)  = \left|  
 \EE_{\otimes^n} \left( \frac{1}{k_n} \sum_{i=1}^{k_n} \eta(X_{(i)})- \eta(X) \right) \left( \mathbf{1}_{\|X_{k_n}-X\|<t} +  \mathbf{1}_{\|X_{k_n}-X\| \geq  t}\right)  \right|.
$$
Using the fact that $\eta$ and $\eta_n$ belong to $[0,1]$, we then have:
$$
\Delta_n(x)  \leq \left|\EE_{\otimes^n} \left( \frac{ \mathbf{1}_{\|X_{(k_n)}-x\|< t}}{k_n} \sum_{i=1}^{k_n} (\eta(X_{(i)})- \eta(x)) \right)  \right| +
\mathbb{P}_{\otimes^n}\left[ \|X_{(k_n)}-x\| \geq t\right].
$$
Now, since $\eta \in \mathcal{C}^{1,0}(\Omega,L)$:
\begin{eqnarray}
\Delta_n(n)  & \leq & \frac{1}{k_n} \sum_{i=1}^{k_n} \EE_{\otimes^n} \left[\mathbf{1}_{\|X_{(k_n)}-x\|<t} \left| \eta(X_{(i)})- \eta(x) \right|\right]\nonumber \\ &&+
\mathbb{P}_{\otimes^n}\left[ \|X_{(k_n)}-x\| \geq t\right], \nonumber \\
& \leq&  \frac{L}{k_n}\EE_{\otimes^n} \left[  \mathbf{1}_{\|X_{(k_n)}-x\|< t}  \sum_{i=1}^{k_n} \left\|X_{(i)}- x \right\| \right] +
\mathbb{P}_{\otimes^n} \left( \|X_{(k_n)}-x\| \geq  t\right), \nonumber \\
& \leq & tL + \mathbb{P}_{\otimes^n} \left( \|X_{(k_n)}-x\| \geq  t\right).
\label{eq:inter_delta}
\end{eqnarray}


\vspace{0.5cm}

\paragraph{Concentration inequality}
We now turn our attention to the control of the last term in the r.h.s. of (\ref{eq:inter_delta}). In the following, for all $x\in \Omega$ and for all $t>0$, $\mu(B(x,t))$ will denote the mass of the ball $B(x,t)$ w.r.t the measure $\mathbb{P}_X$, i.e.
$$ \mu(B(x,t)) := \int_{B(x,t)} \mu(z) dz.$$
Within this context, we sometimes omit the dependency of this quantity w.r.t. the point $x$ and write $\mu_t = \mu(B(x,t))$. Then:
\begin{eqnarray*}
\lefteqn{\mathbb{P}_{\otimes^n}\left( \left\| X_{(k_n)} - x\right\| \geq t \right)},\\
&=&\mathbb{P}_{\otimes^n}\left( \sum_{i=1}^{n}\mathbf{1}_{\{  X_{i}\in B(x,t) \} }\leq k_{n}\right),\nonumber\\
& =& \mathbb{P}_{\otimes^n} \left( \frac{1}{n} \sum_{i=1}^{n} \left[ \mathbf{1}_{\{  X_{i}\in B(x,t) \} }-\mu\left(B(x,t)\right) \right]\leq \frac{k_{n}}{n}-\mu\left(B(x,t)\right)\right),\\
& \leq &  \mathbb{P}_{\otimes^n} \left(\left| \frac{1}{n} \sum_{i=1}^{n} \left[ \mathbf{1}_{\{  X_{i}\in B(x,t) \} }-\mu\left(B(x,t)\right) \right] \right| \geq \mu(B(x,t)) - \frac{k_n}{n} \right),
\end{eqnarray*}
as soon as $\mu(B(x,t)) > \frac{k_n}{n}$. Since $\mathbb{P}_X \in \blabla$ and $x\in K_a$, 
$$ \mu(B(x,t)) \geq \kappa \mu(x) t^d \geq \kappa a t^d.$$
We therefore choose $t$  such that:
\begin{equation} 
 \kappa a t^d \geq 2 \frac{k_n}{n} \Leftrightarrow t\geq \left( 2\frac{k_n}{n} \frac{1}{\kappa a} \right)^{1/d}.
\label{eq:choice_t}
\end{equation}
In particular, with the choice of $t$ given by \eqref{eq:choice_t}, we obtain:
\begin{equation}
\mu(B(x,t)) - \frac{k_n}{n} \geq \frac{\mu(B(x,t))}{2}.
\label{eq:control_mu_t}
\end{equation}

 We then use the following version of the Bennett inequality. 
If $W_{i}$ are random variables such that $W_{i}\leq b$, $v=\sum_{i=1}^{n}\mathbb{E}(W_{i}^{2})$, let $S=\sum_{i=1}^{n}W_i-\mathbb{E}(W_{i})$ then for any $x>0$:
$$
 \mathbb{P}\left(S\geq x\right)\leq \exp\left(-\frac{x^{2}}{2(v+bx/3)}\right).
$$
We apply this version to the random variables  $Z_i=\frac{ \mathbf{1}_{\{  X_{i}\in B(x,t) \} }}{\sqrt{\mu_t(1-\mu_t)}}$. We then have $b=\frac{1}{\sqrt{\mu_t(1-\mu_t)}}$, $v=\frac{n}{1-\mu_t}$ and $x=\frac{n\sqrt{\mu_t}}{2\sqrt{(1-\mu_t)}}$. The exponential bound obtained is then 
$$
\exp\left(-\frac{\frac{n\mu_t}{4(1-\mu_t)}}{2(\frac{1}{1-\mu_t}+\frac{1}{\sqrt{\mu_t(1-\mu_t)}}\frac{\sqrt{\mu_t}}{2\sqrt{(1-\mu_t)}}/3)} \right)=\exp\left(-\frac{3n\mu_t}{28} \right).
$$
Hence:
\begin{eqnarray*}
\lefteqn{\mathbb{P}_{\otimes^n} \left( \left\| X_{(k_n)} - x\right\|  \geq t  \right)}\\
& \leq &  \mathbb{P}_{\otimes^n} \left(\left| \frac{1}{n} \sum_{i=1}^{n} \left[ \mathbf{1}_{\{  X_{i}\in B(x,t) \} }-\mu\left(B(x,t)\right) \right] \right| \geq \frac{\mu(B(x,t))}{2} \right),\\
& \leq &2\exp\left(-\frac{3n\mu_t}{28} \right).
\end{eqnarray*}
Now, using  (\ref{eq:control_mu_t})
 $\mu_t \geq 2k_n /n$, we obtain:
 \begin{eqnarray}\label{eq:last_bound}
\mathbb{P}_{\otimes^n} \left( \left\| X_{(k_n)} - x\right\| \geq t \right)
& \leq &2\exp\left(-\frac{3k_n\mu_t}{14} \right).
\end{eqnarray}
\paragraph{Final Bound}
According to (\ref{eq:inter_delta}) and (\ref{eq:last_bound}), we obtain:
\begin{eqnarray*}
\Delta_n(X) & \leq &  L t +2 \exp \left(- \frac{3k_{n}}{14} \right),\\
& \leq &  L \left(\frac{2}{\kappa}\right)^{1/d}\left( \frac{ k_n}{n a } \right)^{1/d} +  2\exp \left(- \frac{3k_{n}}{14} \right). 
\end{eqnarray*}
\color{black}

This concludes the proof of Proposition \ref{prop:Delta2}.
\end{proof}


\newpage

\section{The smooth discriminant analysis model}
\label{s:sda}

\subsection{Statistical setting}\label{s:poisson}
In this subsection, we focus our attention on an alternative binary classification model. We assume that we have two independent samples $\mathcal{S}_1=(X_1,\dots,X_n)$ and $\mathcal{S}_2=(\tilde X_1,\dots,\tilde X_n)$ of i.i.d. random variables at our disposal, with respective densities $f$ and $g$. We assume that the support of $f$ and $g$ are included in a set $K$. In the following, we can label each element of the sample $\mathcal{S}_1$ (resp. $\mathcal{S}_2$) by $0$ (resp. $1$). 

 In this context, given a new incoming observation, the goal is to predict its corresponding label, namely to determine whether $X\sim f$ or $ X\sim g$. This setting is known as a smooth discriminant analysis model and has been popularized, in particular, by Mammen and Tsybakov (1999).\\

As in the classical binary classification model, a classifier is defined as a measurable function of the samples $\mathcal{S}_1$ and $\mathcal{S}_2$, having values in $\lbrace 0,1 \rbrace$. The risk of each classifier $\Phi_n$ is defined as:
$$ \mathcal{R}(\Phi_n) :=    \frac{1}{2}\left[ \int_{\lbrace x:\Phi_n(x)=1 \rbrace} f(x)dx +  \int_{\lbrace x:\Phi_n(x)=0 \rbrace} g(x)dx \right].$$
It can then be proved that the Bayes Classifier defined as:
$$ \Phi^\star(x) = \mathbf{1}_{\lbrace f(x) \geq g(x) \rbrace},$$
provides the smallest possible risk w.r.t. all possible classifiers. At this step, if we rewrite the Bayes classifier as:
$$ \Phi^\star(x) = \mathbf{1}_{\lbrace \eta(x)\geq \frac{1}{2} \rbrace}, \ \mathrm{where} \ \eta(x) = \frac{f(x)}{f(x)+g(x)},$$
A strong analogy with the classical binary classification problem defined in Section \ref{s:model} (see in particular (\ref{eq:Bayes_classic})) can be observed. The next assumption is equivalent to Assumption $A1$.

\paragraph{Assumption $\mathbf{\widetilde{A}1}$} \textit{An $L$ exists such that $\eta=\frac{f}{f+g}$ belongs to $\mathcal{C}^{1,0}(\Omega,L)$.}

In keeping with the previous study, we work with both a Margin and a Smoothness Assumption on the regression function $\eta$. Once again, the Margin Assumption is quite close to the one introduced for the classical binary classification problem. \\

\paragraph{Assumption $\mathbf{\widetilde{A}2}$} \textit{An $\alpha>0$  and a constant $C>0$ exist such that}

$$
\int_{|\eta-\frac12|<\epsilon} (f+g) \leq C \epsilon^{\alpha} $$.

We also consider the case where the marginal density on $X$ satisfies a Minimal Mass Assumption.
\\
\paragraph{Assumption $\mathbf{\widetilde{A}3}$} \textit{A $\kappa$ exists such that $\mu=\frac{f+g}{2}$ satisfies the condition introduced by $\blabla$.}

\vspace{1cm}
In this context, the principle of the Nearest Neighbor Algorithm introduced in Section \ref{sec:knndef} remains the same. We first aggregate the two samples $\mathcal{S}_1$ and $\mathcal{S}_2$ to obtain $ \mathcal{S} = \left\{ (X_1, 0), \ldots (X_n,0),(\tilde{X}_1,1), \ldots, (\tilde{X}_n,1)\right\}=(\mathcal{X}_i,\mathcal{Y}_i)_{1 \leq i \leq 2n}$. Then, if $\mathcal{X}_{(m)}(x)$ is the $m$-nearest neighbor of $x$ and 
$\mathcal{Y}_{(m)}(x)$ is its label, the K nearest neighbor decision rule is
\begin{equation}\label{eq:kNN_SDA}
\Phi_{n,k}(x)=\left\{\begin{matrix}
1&if& \displaystyle\frac1k\sum_{j=1}^{k}\mathcal{Y}_{(j)}(x)>\frac12,\\
0&otherwise.
\end{matrix}\right.
\end{equation}

\subsection{Case of bounded from below densities}

The following theorem provides a control on its corresponding minimax excess risk when the underlying density on $X$ is bounded from below by a strictly positive constant.

\begin{theo}\label{theo:Main1bis} Assume that Assumptions $\mathbf{\tilde{A}1-\tilde{A}3}$ hold and that the density of $X$ is lower bounded by $\mu_{-}>0$. If $k_n = \lfloor n^{\frac{2}{2+d}}\rfloor$
, then
$$
\left[ \mathcal{R}(\Phi_{n,k_n})- \mathcal{R}(\Phi^{*}) \right] \lesssim \left(\frac{\log n}{n}\right)^{\frac{1+\alpha}{2+d}} .
$$
\end{theo}

We can immediately observe that we lose a log term compared to the minimax rate obtained in Section \ref{s:rates} for the  classification model.

According to our knowledge, the optimal rates in the present model (that is, the smooth discriminant analysis) have never been investigated under Assumptions $\mathbf{\widetilde{A}1-\widetilde{A}3}$. Nevertheless, it appears that under some slightly different assumptions (see \cite{Tsybakov_tout_seul} and \cite{AudTsy}), the minimax risk of the smooth discriminant analysis problem is the same as the one of the classification model. It is therefore reasonable to think that our rate is near-optimal (optimal up to a log term).


From a technical point of view, the main counterpart when using the smooth discriminant analysis setting is that conditionally to the spatial positions of the ordered  aggregate sample $(X_{(j)}(x))_{1 \leq j \leq 2n}$, the corresponding labels $(Y_{(j)}(x))_{1 \leq j \leq 2n}$ are no longer independent of each other. This is a significant difference with the classification model considered in Section \ref{s:model} and makes it quite difficult to obtain a concentration inequality for the empirical ``regression" function:
$$
\eta_{n}(x) = \frac{1}{k_n} \sum_{j=1}^{k_n} Y_{(j)}(x).
$$
In order to get around this problem, we adopt a Poisson embedding (see Subsection \ref{s:poisson_tech} for further details), which makes it possible to satisfy a
control of the excess risk, up to some additional logarithmic terms. We assume that such a logarithmic term may be removed using some concentration inequalities on negatively associated random variables. In fact, we hypothesize that for asymptotically large $n$, the random sequence $(\sum_{j=1}^p Y_{(j)}(x))_{1 \leq p \leq k_n}$ is negatively associated as soon as $k_n/n \longmapsto 0$, which may make it possible to remove the logarithmic term (see \textit{e.g.} \cite{Kleinphd},\cite{MR1777538}).

\subsection{Case of general densities}
We provide  a short paragraph here about the case of general densities and introduce the Tail Assumption necessary to derive a uniform rate of consistency for the nearest neighbor rule.

\begin{theo}\label{theo:Main2bis} Assume that Assumptions $\mathbf{\widetilde{A}1-\widetilde{A}3}$ hold and that a function $\psi$ exists such that $\mathbb{P}_X \in \PTf$. Then:
$$
\left[ \mathcal{R}(\Phi_n)- \mathcal{R}(\Phi^{*}) \right] \lesssim \left(\frac{\log n}{n}\right)^{\frac{1+\alpha}{3+\alpha+d}} \qquad \text{if} \qquad \psi=Id.
$$
Otherwise:
$$
\left[ \mathcal{R}(\Phi_n)- \mathcal{R}(\Phi^{*}) \right] \lesssim \nu_{n,\alpha,d}^{1+\alpha} \quad \text{where} \quad \frac{\log n}{n} = \psi^{-1}(\{\nu_{n,\alpha,d}\}^{1+\alpha}) \{\nu_{n,\alpha,d}\}^{2+d}.$$
\end{theo}

We only provide the proof of Theorem \ref{theo:Main1bis}: the one of Theorem \ref{theo:Main2bis}
relies on a mixed association of some arguments in Theorem \ref{theo:Main2} and Theorem \ref{theo:Main1bis}.

\subsection{Proof of Theorem \ref{theo:Main1bis}}\label{s:poisson_tech}
A straightforward decomposition yields the following proposition.
\begin{prop}\label{prop:risk_dec}
For any classifier $\Psi$, we have:
$$
\mathcal{R}(\Psi) - \mathcal{R}(\Phi^*) = \int_{\Psi \neq \Phi^*} \left| 2 \eta(x)-1\right| \frac{f(x)+g(x)}{2} dx.
$$
\end{prop}
\begin{proof}[Proof of Theorem \ref{theo:Main1bis}]
The  beginning of the proof is similar to the one of Theorem \ref{theo:Main1}. We use Proposition \ref{prop:risk_dec} and
Assumption $\mathbf{\widetilde{A}3}$ to obtain:
\begin{eqnarray}\label{eq:excess_sda}
\mathcal{R}(\Phi_n) - \mathcal{R}(\Phi^*)&  \leq & 2 C c_2 \epsilon^{1+\alpha}
\nonumber \\ &&\hspace{-4em}+ c_2 \epsilon \sum_{j \geq 1}2^{j}
\int_{K\cap \{ 0 < |\eta(x)-1/2| < 2^{j} \epsilon\}} \PP \left( |\eta_n(x) - \eta(x)| > 2^{j-1} \epsilon \right)  dx.
\end{eqnarray}
At this step, the major difference between the discriminant analysis and the classification model is the control of the deviation inequality on the right hand side of the equation above. Indeed, conditionally to $((\mathcal{X}_{(i)}(x))_{1 \leq i \leq 2n}$, the corresponding labels $((\mathcal{Y}_{(i)}(x))_{1 \leq i \leq 2n}$ are no longer independent.

We use  Proposition \ref{prop:poisson} with:
\begin{equation}\label{eq:choice_eps_da}
\epsilon_n = 2 C \left( \frac{k_n}{n} \right)^{1/d},\end{equation}
 where $C$ is the constant that appears in Proposition \ref{prop:poisson} $ii)$. We then obtain:
 \begin{eqnarray*}
 \lefteqn{
\PP \left( |\eta_n(x) - \eta(x)| > 2^{j-1} \epsilon_n \right) } \\
 &\leq& 2 \pi n \left[ \exp\left(- 2 k_n  \left((2^{j-1}-1/2) \epsilon_n\right)^2\right) + 
\mathbf{1}_{\left\{ 2^{j} \left(\frac{k_n}{n}\right)^{1/d} \leq C^{-1} \right\}} e^{- n} 
 \right].
\end{eqnarray*}
The suitable choice of $k_n$ and $\epsilon_n$ then becomes:
\begin{equation}\label{eq:choice_eps_kn}
k_n \epsilon_n^2 \sim a \log(n)
\end{equation}
for a sufficiently large universal constant $a>0$. Hence,   \eqref{eq:choice_eps_da} and \eqref{eq:choice_eps_kn} yields:
$$
\epsilon_n \sim \left( \frac{\log(n)}{n}\right)^{\frac{1}{1+d}} \qquad \text{and} \qquad k_n \sim \log(n)^{d/(2+d)} n^{2/(2+d)}.$$
This last choice implies:
\begin{equation}\label{eq:ouf?}
\PP \left( |\eta_n(x) - \eta(x)| > 2^{j-1} \epsilon_n \right)  \lesssim \exp(- 2 a 2^{2j})  + 2 \pi n e^{-n} \mathbf{1}_{\left\{ 2^{j} \leq C^{-1} \left(\frac{n}{k_n}\right)^{1/d} \right\}} 
\end{equation}
Plug in \eqref{eq:ouf?} in \eqref{eq:excess_sda} allows us to conclude:
\begin{eqnarray*}
\lefteqn{\mathcal{R}(\Phi_n) - \mathcal{R}(\Phi^*)}\\
& \leq & 2 C c_2 \epsilon_n^{1+\alpha} + c_2 \epsilon_n^{1+\alpha} \sum_{j \geq 1} 2^{j(1+\alpha)+1} \exp(- \tilde{C} 2^{2j}) \\
& & + \epsilon_n^{1+\alpha} 2 \pi n e^{-n} \sum_{j : 2^j \leq C^{-1} \left(\frac{n}{k_n}\right)^{1/d} }2^{j(1+\alpha)} \\
& \lesssim & \epsilon_n^{1+\alpha} + 2 \pi n e^{-n} \left( \frac{n}{k_n} \right)^{(1+\alpha)/d},
\end{eqnarray*}
where the last inequality follows from standard upper bounds on geometrical sums. We then deduce:
$$
 \sup_{F\in\mathcal{F}} \left[ \mathcal{R}(\Phi_n)- \mathcal{R}(\Phi^{*}) \right] \leq C \left( \frac{\log(n)}{n}\right)^{\frac{1+\alpha}{2+d}} .
 $$
\end{proof}

\begin{prop}[Concentration of $\eta_n$ with Poisson approximation]\label{prop:poisson}
In the smooth discriminant analysis model, assume that $k_n < n$. We then have:
\begin{itemize}\item[i)] For any $t \geq 0$:
$$
\PP \left( \left| \eta_n(x) - \EE \eta_{n}(x)  \right|> t \right) \leq 2 \pi n \left[ 2 \exp (- 2 k_n t^2)  +\mathbf{1}_{\{ |t| \leq 1 \}} e^{- n} \right].
$$
\item[ii)] For any $t \geq 0$ and if $k_n>>\log(n)$, there exists a constant $C$ such that:
\begin{eqnarray*}
\lefteqn{\PP \left( \left| \eta_n(x) - \eta(x)  \right|> t \right)}\\&
 &\leq 2 \pi n \left[
 2 \exp \left(
 - 2 k_n \left[
 t - C \left(\frac{k_n}{n}\right)^{1/d}  \right]^2
 \right) +\mathbf{1}_{\{ |t| \leq 1 \}} e^{- n}
   \right].
\end{eqnarray*}
\end{itemize}
\end{prop}
\begin{proof}
\underline{Poissonization:}\\
In order to eliminate the dependency between the ordered statistics $(\mathcal{X}_{(i)} )_{1 \leq i \leq 2n}$, we use an idea introduced by \cite{Kac} and randomize the size of the sample by using some Poisson random variables.

First consider $(N_1,N_2)$ two independent random variables following a Poisson distribution $\mathcal{P}(n)$ as well as a $N \sim \mathcal{P}(2n)$ independent of $(N_1,N_2)$. We now build an artificial sample of size $N_1+N_2$ divided into two parts:
$$\mathcal{S}^{\mathcal{P}}_{0-1}= (X_i,Y_i)_{1 \leq i \leq N_1+N_2} \sim(g d\lambda\otimes \delta_0)^{\otimes N_1}  \bigotimes \,  (f d\lambda\otimes \delta_1)^{\otimes N_2}.$$
Then, if $\sigma$ denotes a random permutation picked uniformly in $\mathfrak{S}_{N_1+N_2}$ and independent of the previous realizations, the permutated sample:
$$
\mathcal{S}^{\mathcal{P},\sigma}_{0-1} = (X_{\sigma(i)},Y_{\sigma(i)})_{1 \leq i \leq N_1+N_2},
$$
follows the same distribution as the sample of i.i.d. realizations:
$$
\mathcal{S}^{\mathcal{P}}_{\eta}=(U_i,V_i)_{1 \leq i \leq N},
$$
where $U \sim \frac{f+g}{2} d \lambda$ and $V \vert U \sim \mathcal{B}(\eta(U))$ with $\eta(U)=\frac{f(U)}{f(U)+g(U)}$.\\
Recall that $\eta_n$ is built from our original sample $\mathcal{S}$ according to:
$$
\eta_n(x)=\frac{1}{k_n}\sum_{j=1}^{k_n}\mathcal{Y}_{(j)}(x).
$$
Our aim is to study the deviation of $\eta_n$ from its mean. For this purpose, we introduce
its   Poisson counterpart built from $\mathcal{S}^{\mathcal{P},\sigma}_{0-1}$:
$$
\eta_n^{N_1,N_2}(x)=\frac{1}{k_n}\sum_{j=1}^{k_n\wedge (N_1+N_2)}Y_{(j)}(x),
$$
which  is independent of the random choice of $\sigma$. At last, we define $\eta_n^{N}$ by
$$
\eta_n^{N}=\frac{1}{k_n}\sum_{j=1}^{k_n\wedge N}V_{(j)}(x).
$$
where we have used the convention $\sum_{j\geq 1}^{0} a_j = 0$ for any sequence $(a_j)_{j \geq 1}$.
\underline{Proof of $i)$.}
Since $\mathcal{L}\left(\eta_n(x) \right)= \mathcal{L}\left( \eta_n(x)^{N_1,N_2} \vert N_1=N_2=n\right)$, we deduce that
\begin{eqnarray}
\PP \left( \left| \eta_n(x) - \EE \eta_{n}(x)  \right|> t \right)& =& \PP \left( \left| \eta^{N_1,N_2}_n(x) - \EE \eta^{N_1,N_2}_{n}(x)  \right|> t \vert N_1=N_2=n \right)  \nonumber\\
 & = &  \frac{\PP \left( \left| \eta^{N_1,N_2}_n(x) - \EE \eta^{N_1,N_2}_{n}(x)  \right|> t , N_1=N_2=n \right) }{\PP(N_1=N_2=n)} \nonumber\\
  & \leq & \frac{\PP \left( \left| \eta^{N_1,N_2}_n(x) - \EE \eta^{N_1,N_2}_{n}(x)  \right|> t \right)}{\PP(N_1=N_2=n)}.\label{eq:poisson_bound}
\end{eqnarray}
Again, we can observe that $\mathcal{L}\left(\eta_n^{N_1,N_2}(x) \right)= \mathcal{L}\left( \eta_n^N(x)\right)$ and
\begin{eqnarray*}
\PP \left( \left| \eta^{N_1,N_2}_n(x) - \EE \eta^{N_1,N_2}_{n}(x)  \right|> t \right) &=& \PP \left( \left| \eta^{N}_n(x) - \EE \eta^{N}_{n}(x)  \right|> t \right) \\
& = & \PP \left( \left| \eta^{N}_n(x) - \EE \eta^{N}_{n}(x)  \right|> t  ,  N> k_n\right)  \\ 
&&+ \PP \left( \left| \eta^{N}_n(x) - \EE \eta^{N}_{n}(x)  \right|> t  , N< k_n\right) .
\end{eqnarray*}
We now study the two terms of the upper bound separately.\\

\paragraph{Upper bound of $ \PP \left( \left| \eta^{N}_n(x) - \EE \eta^{N}_{n}(x)  \right|> t  , N> k_n\right) $}
We can use the standard Hoeffding inequality:
$$
\PP \left( \left| \eta^{N}_n(x) - \EE \eta^{N}_{n}(x)  \right|> t  \vert N> k_n\right) \leq 2 \exp (- 2 k_n t^2),
$$
and trivially bounding  $\PP(N>k_n)$ by $1$ yields:
\begin{equation}\label{eq:tech1}
 \PP \left( \left| \eta^{N}_n(x) - \EE \eta^{N}_{n}(x)  \right|> t, N> k_n\right)  \leq 2 \exp(- 2 k_n t^2).
\end{equation}

\paragraph{Upper bound of $ \PP \left( \left| \eta^{N}_n(x) - \EE \eta^{N}_{n}(x)  \right|> t  , N< k_n\right) $} 
From the Chernoff bound of the left Poisson tail, we obtain:
$$
\PP(N<k_n) \leq \frac{e^{-2n} (e 2n)^{k_n}}{k_n^{k_n}} = e^{-2n \left( 1- \frac{k_n}{2n} \log\left(  \frac{e k_n}{n}\right) \right)} \leq e^{- n},
$$
as soon as $k_n < n$.  Moreover, we have:
\begin{eqnarray*}
\lefteqn{\PP \left( \left| \eta^{N}_n(x) - \EE \eta^{N}_{n}(x)  \right|> t  \vert  N< k_n\right)}\\ &=& \PP \left( \frac{N}{k_n} 
\frac{\sum_{j=1}^N \left[ V_{(j)}(x) - \eta(U_{(j)}(x) \right]}{N} > t  \vert k_n > N \right) \\
& \leq & \mathbf{1}_{\{ |t| \leq 1 \}}.
\end{eqnarray*}
It follows that:
\begin{equation}\label{eq:tech2}
\PP \left( \left| \eta^{N}_n(x) - \EE \eta^{N}_{n}(x)  \right|> t  , N< k_n\right)  \leq \mathbf{1}_{\{ |t| \leq 1 \}}  e^{-n}.
\end{equation}

We then deduce from \eqref{eq:poisson_bound}, \eqref{eq:tech1} and \eqref{eq:tech2} that:
$$
\PP \left( \left| \eta_n(x) - \EE \eta_{n}(x)  \right|> t \right) \leq \frac{2 \exp (- 2 k_n t^2)  +\mathbf{1}_{\{ |t| \leq 1 \}} e^{-n}   }{\PP(N_1=N_2=n)}.
$$
The Stirling formula concludes the proof of $i)$:
$$
\PP \left( \left| \eta_n(x) - \EE \eta_{n}(x)  \right|> t \right) \leq 2 \pi n \left( 2 \exp (- 2 k_n t^2)  +\mathbf{1}_{\{ |t| \leq 1 \}} e^{- n} \right)
$$

\underline{Proof of $ii)$.} We now build the decomposition already used in the proof of Theorem \ref{theo:Main1}.
\begin{eqnarray*}
\lefteqn{\PP \left( \left| \eta_n(x) -  \eta(x)  \right|> t \right)}\\ &= &
\PP \left( \left| \eta^{N_1,N_2}_n(x) -  \eta(x)  \right|> t  \vert N_1=N_2=n\right)  \\
& \leq & \frac{\PP \left( \left| \eta^{N_1,N_2}_n(x) -  \eta(x)  \right|> t \right)}{\PP(N_1=N_2=n)} \\
& \leq & \frac{\PP \left( \left| \eta^{N_1,N_2}_n(x) -  \eta(x)  \right|> t , N_1+N_2> k_n\right)}{\PP(N_1=N_2=n)} \\
& & + 
\frac{\PP \left( \left| \eta^{N_1,N_2}_n(x) -  \eta(x)  \right|> t , N_1+N_2\leq  k_n\right)}{\PP(N_1=N_2=n)}\\
& \leq &  \frac{\PP \left( \left| \eta^{N}_n(x) - \EE  \eta^{N}_n(x) \right| > t - \left| \EE   \eta^{N}_n(x) -  \eta(x)  \right|
\vert N>k_n
\right) \PP(N>k_n)}{\PP(N_1=N_2=n)} \\
&& +  \frac{\mathbf{1}_{\{ |t| \leq 1 \}} e^{- n} }{\PP(N_1=N_2=n)}.\\
\end{eqnarray*}
We can slightly modify Proposition \ref{prop:Delta2} to obtain that for any $m>k_n$:
$$
\left| \EE \left[  \eta^{N}_n(x) -  \eta(x) \vert N=m \right] \right|\leq C \left( \left(\frac{k_n}{m}\right)^{1/d} + e^{- 3 k_n/14} \right).
$$
Now if we set $m_n = 2n-(2n)^{\beta}$ for $\beta < 1$, we can write:
\begin{eqnarray*}
\lefteqn{\left| \EE  \left[ \eta^{N}_n(x) -  \eta(x)\right] \right| }\\ & \leq & \PP(N<m_n) +   \left| \sum_{m=m_n}^{+\infty}\EE \left[ \eta^{N}_n(x) -  \eta(x) \vert N=m \right]  \PP(N=m) \right|\\
& \leq & \PP(N<m_n) + C \PP(N> m_n )  \left( \left( \frac{k_n}{m_n}\right)^{1/d} + e^{-3 k_n/14} \right)
\\ &\leq & \PP(N<m_n) + C \left( \left( \frac{k_n}{n}\right)^{1/d} + e^{-3 k_n/14} \right).
\end{eqnarray*}
where we have used $m_n> n$ for $n$ large enough. Again, the Chernoff bound on the Poisson tail yields
$$
 \PP(N<m_n) \leq  e^{-n^{2\beta-1}}.
$$
Any choice of $\beta \in (\frac{1}{2},1)$ implies
$$
\left| \EE  \left[ \eta^{N}_n(x) -  \eta(x)\right] \right| \lesssim \left( \frac{k_n}{n} \right)^{1/d}.
$$ 

The end of the proof is now straightforward using the bounds already given in the proof of $i)$.
\end{proof}

\end{document}